\documentclass[reqno]{amsart}

\usepackage[utf8]{inputenc}
\usepackage{amsmath}
\usepackage{amssymb}
\usepackage{hyperref}
\usepackage{verbatim}
\usepackage{dsfont}
\usepackage{enumerate}
\usepackage[all]{xy}
\usepackage[mathscr]{eucal}
\usepackage[usenames]{color}
\usepackage{amsthm}
\usepackage{stmaryrd}
\usepackage{amsfonts}
\usepackage{latexsym}
\usepackage{amscd}
\usepackage{graphicx}
\usepackage{setspace}
\usepackage{wasysym}
\usepackage[usenames,dvipsnames]{xcolor}
\usepackage{tikz}
\usepackage{tikz-cd}
\usepackage{todonotes}

\makeatletter

\def\bd#1{\text{\boldmath${#1}$}}

\newtheorem{theorem}{Theorem}[section]
\newtheorem{iTheorem}{Theorem}
\newtheorem{lemma}[theorem]{Lemma}
\newtheorem{proposition}[theorem]{Proposition}
 
\theoremstyle{definition}  

\newtheorem{example}[theorem]{Example}

\newtheorem{remark}[theorem]{Remark}

\@addtoreset{equation}{section}

\newcommand{\End}{\operatorname{End}}
 
\newcommand{\Hom}{\operatorname{Hom}} 
 
\newcommand{\Rep}{\operatorname{Rep}}
\newcommand{\Res}{\operatorname{Res}}

\newcommand{\Z}{\mathbb{Z}}

\newcommand{\C}{\mathbb{C}}

\newcommand{\arup}[1]{\stackrel{#1}{\longrightarrow}}

\newcommand{\unit}{\mathds{1}}

\newcommand{\un}{\eta}
\newcommand{\coun}{\varepsilon}
\newcommand{\mult}{\mu}
\newcommand{\comult}{\Delta}
\newcommand{\symm}{ s}
\newcommand{\jelly}{ j}
\newcommand{\Stirling}{\mathsf{S}}
\newcommand{\Bell}{\mathsf{B}}

\newcommand{\B}{\mathcal{B}}
\newcommand{\JB}{\mathcal{JB}}
\renewcommand{\P}{\mathcal{P}}
\newcommand{\JP}{\mathcal{JP}}

\renewcommand{\k}{\Bbbk}

\newcommand{\jellycolor}{blue}

\newcommand{\crosscap}{\mathord{
\begin{tikzpicture}[baseline = 0pt]
        \draw[-,thick] (0,-.01) to[out=up, in=left] (0.1,0.2) to[out=right, in=up] (0.2,-.01);
        \draw[-,thick] (0.1,-.01) to[out=up, in=left] (0.2,0.2) to[out=right, in=up] (0.3,-.01);
\end{tikzpicture}
}}

\allowdisplaybreaks

\newcommand{\shift}{7}

\begin{document}

\title{Jellyfish partition categories}

\author[Comes]{Jonathan Comes}
%\address{J.C.: Boise, Idaho}
\email{jonnycomes@gmail.com}

\date{\today}

%\subjclass[2010]{17B10, 18D10}
%\thanks{2010 {\it Mathematics Subject Classification}: 17B10, 18D10.}

\begin{abstract}
For each positive integer $n$, we introduce a monoidal category $\JP(n)$ using a generalization of partition diagrams. When the characteristic of the ground field is either 0 or at least $n$, we show $\JP(n)$ is monoidally equivalent to the full subcategory of $\Rep(A_n)$ whose objects are tensor powers of the natural $n$-dimensional permutation representation of the alternating group $A_n$. 
\end{abstract}

\maketitle  

\section{Introduction}\label{section: Introduction}

Let $\k$ denote a field with $\operatorname{char}\k\not=2$. Classical Schur-Weyl duality, a staple in representation theory,  concerns the commuting actions of the general linear group $GL_n(\k)$ and the symmetric group $S_k$ on the space $V^{\otimes k}$ where $V=\k^n$. In particular, it says the map $\k[S_k]\to\End_{GL_n(\k)}(V^{\otimes k})$ is surjective for all $k$ and $n$. If we replace $GL_n(\k)$ with its subgroup $S_n$ of all permutation matrices, then the analog of the group algebra $\k[S_k]$ is the partition algebra $P_k(n)$ introduced by Martin \cite{Martin1991}. Indeed, Martin's partition algebras are equipped with an action on $V^{\otimes k}$ which commutes with the action of $S_k$, and the maps $P_k(n)\to\End_{S_n}(V^{\otimes k})$ are always surjective. This result easily extends to maps between arbitrary (not necessarily equal) tensor powers of $V$. More precisely, there is a $\k$-linear monoidal category $\P(n)$, the \emph{partition category} (see \S\ref{subsection: the category P}), which is equipped with a full monoidal functor $\Phi:\P(n)\to\Rep(S_n)$ (see \S\ref{subsection: the functor Phi}) and the image of $\Phi$ is precisely the full subcategory of $\Rep(S_n)$ with objects $V^{\otimes k}$ for all $k\geq 0$. 

In this paper we describe a \emph{jellyfish partition category} $\JP(n)$ which plays the role of $\P(n)$ when we replace $S_n$ with the alternating group $A_n$. The category $\JP(n)$ is a $\k$-linear monoidal category defined by adding one extra generator and three extra relations to a known presentation of the partition category (see \S\ref{definition of JP}). In particular, $\JP(n)$ is equipped with a monoidal functor $\P(n)\to\JP(n)$ which allows us to view partition diagrams (morphisms in $\P(n)$) as morphisms in $\JP(n)$. We give a complete diagrammatic description of morphisms in $\JP(n)$ by associating a ``jellyfish diagram" to the additional generator (see \S\ref{JP diagrams}). The category $\JP(n)$ is also equipped with a monoidal functor $\Psi:\JP(n)\to\Rep(A_n)$, which makes the following diagram commute:
\begin{equation}\label{square of functors}
    \begin{tikzcd} 
    \P(n) \arrow{rr}{\Phi}\arrow{d} && \Rep(S_n)\arrow{d}{\Res^{S_n}_{A_n}}\\ 
    \JP(n) \arrow{rr}{\Psi} && \Rep(A_n)
    \end{tikzcd}
\end{equation}
The first of two main results of this paper is the following:
\begin{iTheorem}\label{Psi is full}
    The functor $\Psi$ is full.   
\end{iTheorem}
In particular, we get a version of Schur-Weyl duality between the actions of the alternating group $A_n$ and the jellyfish partition algebra $JP_k(n):=\End_{\JP(n)}(k)$ on the space $V^{\otimes k}$. Our second main result requires some restrictions of the characteristic of the ground field:

\begin{iTheorem}\label{Psi is faithful}
    If $n>1$ and either $\operatorname{char}\k=0$ or $\operatorname{char}\k\geq n$, then $\Psi$ is faithful. Hence, $\JP(n)$ is monoidally equivalent to the full subcategory of $\Rep(A_n)$ whose objects are tensor powers of the natural $n$-dimensional permutation representation.  
\end{iTheorem}

We note that this theorem does not have an analog for partition categories. In fact, $\Phi$ is never faithful (see part (2) of Theorem \ref{thm: Phi is full and kernel}). However, since the restriction functor $\Res^{S_n}_{A_n}$ is faithful, the commutative diagram (\ref{square of functors}) shows that the kernel of $\Phi$ is precisely the kernel of the functor $\P(n)\to\JP(n)$. In other words, the three defining relations of $\JP(n)$ are enough to see the entire kernel of $\Phi$.  

\subsection{Jellyfish Brauer categories}\label{jellyfish Brauer} 
Just as the partition category admits a full functor to the category of representations of the symmetric group, the so-called Brauer category $\B(n)$ is equipped with a full functor to the category of representations of the orthogonal group $O(n)$ (see for instance \cite{LZ}). There is also a \emph{jellyfish Brauer category} $\JB(n)$, which can be defined by taking a known presentation of $\B(n)$ as a $\k$-linear monoidal category (see \cite[Theorem 2.6]{LZ} or \cite[\S1.1]{BE}) and adding to it a new generator (an $n$-legged jellyfish) and three new relations (\ref{jellyfish symmetry})--(\ref{jellyfish relation}). More concisely, $\JB(n)$ is the free $\k$-linear symmetric monoidal category generated by a self-dual object of dimension $n$ which admits a skew-symmetric $n$-form $\jelly$ satisfying (\ref{jellyfish relation}). The category $\JB(n)$ is equipped with a monoidal functor to the category of representations of the special orthogonal group $SO(n)$. Quite recently, Lehrer-Zhang showed that this functor is fully faithful when $\k=\C$ \cite[Theorem 6.1]{LZ2016}. In particular, the corresponding \emph{jellyfish Brauer algebra} $JB_k(n):=\End_{\JB(n)}(k)$ is isomorphic to $\End_{SO(n)}(V^{\otimes k})$ where $V$ is the natural $n$-dimensional representation of $SO(n)$. The desire for such a diagram algebra was the motivation for the work that produced this paper. The existence of such a diagram algebra was eluded to by Brauer in \cite{Brauer}. Brauer proposed a diagram algebra consisting of the usual Brauer diagrams along with Brauer diagrams with $n$ unconnected dots. Translating these diagrams into jellyfish Brauer diagrams amounts to adding a jellyfish whose $n$ legs, read from left to right, are connected to the $n$ unconnected dots, reading the dots from left to right along the top row and then along the bottom row. For example, here is one of Brauer's diagrams alongside the corresponding jellyfish diagram in $JB_5(4)$:
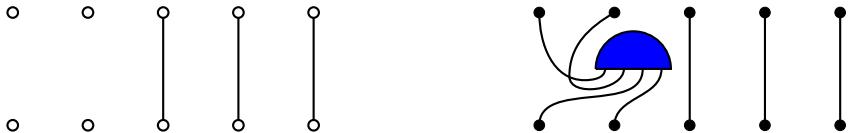
\[
    \begin{tikzpicture}[baseline = 12pt]
    %Draw edges:
    \draw[-,thick] (5,0) to (5,1.5);
    \draw[-,thick] (4,0) to (4,1.5);
    \draw[-,thick] (3,0) to (3,1.5);
    %Draw bottom vertices
    \draw \foreach \m in {1,...,5} {
        (\m,0) node[circle, draw, fill=white, thick, inner sep=0pt, minimum width=4pt]{}
        };
    %Draw top vertices
    \draw \foreach \m in {1,...,5} {
        (\m,1.5) node[circle, draw, fill=white, thick, inner sep=0pt, minimum width=4pt]{}
        };
    %Draw bottom vertices
    \draw \foreach \m in {1,...,5} {
        (\m+\shift,0) node[circle, draw, fill=black,inner sep=0pt, minimum width=4pt]{}
        };
    %Draw top vertices
    \draw \foreach \m in {1,...,5} {
        (\m+\shift,1.5) node[circle, draw, fill=black,inner sep=0pt, minimum width=4pt]{}
        };
    %Draw edges:
    \draw[-,thick] (5+\shift,0) to (5+\shift,1.5);
    \draw[-,thick] (4+\shift,0) to (4+\shift,1.5);
    \draw[-,thick] (3+\shift,0) to (3+\shift,1.5);
    %Draw jellyfish
    \filldraw[fill=\jellycolor,-,thick] (1.75+\shift,0.75) to (2.75+\shift,0.75) to[out=up, in=right] (2.25+\shift,1.25) to[out=left, in=up] (1.75+\shift,0.75);
    \draw[-,thick] (1.75+1/8+\shift,0.75) to[out=down, in=right] (1.6+\shift,0.6) to[out=left, in=down] (1+\shift,1.5);
    \draw[-,thick] (1.75+3/8+\shift,0.75) to[out=down, in=down] (1.4+\shift,0.65) to[out=up, in=-150] (2+\shift,1.5);
    \draw[-,thick] (1.75+5/8+\shift,0.75) to[out=down, in=up] (1+\shift,0);
    \draw[-,thick] (1.75+7/8+\shift,0.75) to[out=down, in=up] (2+\shift,0);
\end{tikzpicture}
\]
Brauer did not give a rule for multiplying his diagrams with unconnected vertices. Instead, he wrote ``The rule for the multiplication $\ldots$ can also be formulated. It is, however, more complicated and shall not be given here." A multiplication rule for Brauer's diagrams was formulated in \cite{Grood}. The resulting \emph{even Brauer algebras} were studied further in \cite{Nebhani}. The even Brauer algebras have a basis given by Brauer's proposed diagrams, however the multiplication rule is not associative. As a consequence of \cite[Theorem 6.1]{LZ2016}, the jellyfish Brauer algebras are associative quotients of the even Brauer algebras, with the quotient map given by adding jellyfish in the manner prescribed above (compare with Remark \ref{Grood remark}).

\subsection{Outline}\label{subsection: outline}
In \S\ref{section: The partition category} we give an exposition of the definitions and properties of the partition category which are relevant to this paper. In particular, we give both a diagrammatic description of $\P(n)$ in terms of partition diagrams as well as a presentation of $\P(n)$ as a $\k$-linear symmetric monoidal category in terms of generators and relations. We close \S\ref{section: The partition category} by proving the functor $\Phi$ mentioned above is full, and giving an explicit description of its kernel. In \S\ref{section: JP} we define the category $\JP(n)$ in terms of generators and relations, and then develop a diagrammatic calculus for its morphisms in terms of jellyfish partition categories. In \S\ref{Psi} we study the action of the alternating group $A_n$ on $V^{\otimes n}$, which allows us to prove Theorem \ref{Psi is full} in \S\ref{proof of theorem A}. We close that section in \S\ref{counting dimensions} with a description of the dimensions of relevant Hom-spaces in terms of Stirling and Bell numbers. The whole of \S\ref{section: Psi is faithful} is devoted to the proof of Theorem \ref{Psi is faithful}.

\subsection{Acknowledgments} 
I would like to thank Jonathan Kujawa for initiating this project by pointing out the paper \cite{Grood}, and for several useful conversations since. Part of this project was completed while I enjoyed a visit to the Max Planck Institute in Bonn. I would like to thank the institute for their hospitality.

\section{The partition category}\label{section: The partition category}

In this section we review the properties of the partition category relevant to this paper. The results here are not new, but proofs are provided since the majority of the content will be needed in upcoming sections. We start with the definition of the partition category, following the analogous treatment of partition algebras found in \cite{HR}.  

\subsection{Partition diagrams}\label{subsection: Partition diagrams}
Given $k,\ell\in\Z_{\geq0}$,  a \emph{partition diagram of type $k\to \ell$} is a simple graph whose vertices are labeled by $$\{i~|~1\leq i\leq k\}\cup\{i'~|~1\leq i\leq \ell\}.$$ For example, here is a partition diagram of type $7\to 5$:
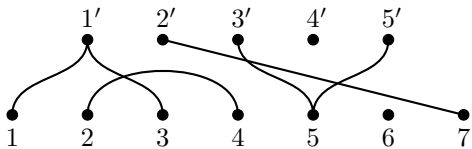
\tikzstyle{every node}=[inner sep=0pt, minimum width=4pt]
\[
    \begin{tikzpicture}[baseline = 12pt]
        %Draw bottom vertices
        \draw \foreach \m in {1,...,7} {
            (\m,0) node[circle, draw, fill=black]{}
            (\m,-.3) node{\m} %labels
            };
            %Draw top vertices
            \draw \foreach \m in {1,...,5} {
            (\m+1,1) node[circle, draw, fill=black]{}
            (\m+1.05,1.3) node{$\m'$} %labels
            };
            %Draw edges:
            \draw[-,thick] (1,0) to[out=up, in=down] (2,1) to[out=down, in=up] (3,0);
            \draw[-,thick] (2,0) to[out=up,in=up] (4,0);
            \draw[-,thick] (4,1) to[out=down, in=up] (5,0) to[out=up, in=down] (6,1);
            \draw[-,thick] (3,1) to (7,0);
    \end{tikzpicture}
\]
We write $D:k\to \ell$ to indicate that $D$ is a partition diagram of type $k\to \ell$. As in the example above, our partition diagrams will always be drawn with two horizontal rows of vertices such that $1,\ldots,k$ are below $1',\ldots,\ell'$. With this convention in mind, we will omit the labels in partition diagrams for remainder of the paper.  The connected components of a partition diagram prescribe a partition of the set of vertices into mutually disjoint nonempty subsets, which we will refer to as \emph{parts}. We say that two partition diagrams are \emph{equivalent} if they give rise to the same set partition. For example, the following partition diagram is equivalent to the one pictured above:
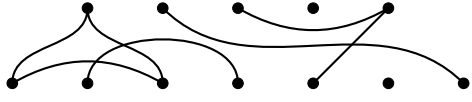
\[
    \begin{tikzpicture}[baseline = 12pt]
        %Draw bottom vertices
        \draw \foreach \m in {1,...,7} {
            (\m,0) node[circle, draw, fill=black]{}
            };
            %Draw top vertices
            \draw \foreach \m in {1,...,5} {
            (\m+1,1) node[circle, draw, fill=black]{}
            };
            %Draw edges:
            \draw[-,thick] (1,0) to[out=up, in=down] (2,1) to[out=down, in=up] (3,0) to[out=150, in=30] (1,0);
            \draw[-,thick] (2,0) to[out=up,in=up] (4,0);
            \draw[-,thick] (5,0) to (6,1);
            \draw[-,thick] (4,1) to[out=-30,in=-150] (6,1);
            \draw[-,thick] (3,1) to[out=-45,in=135] (7,0);
    \end{tikzpicture}
\]

We put a partial order on the set of all equivalence classes of partition diagrams of type $k\to \ell$ by declaring $D_1\leq D_2$ whenever the set partition corresponding to $D_2$ is coarser than that of $D_1$. In other words, $D_1\leq D_2$ if $D_1$ can be obtained by removing edges from $D_2$.  

Given partition diagrams $D_1:k\to \ell$ and
$D_2:\ell\to m$, we can stack $D_2$ on top of $D_1$ to 
obtain a graph $D_2\atop D_1$ with three rows of vertices. Let $\beta(D_1,D_2)$ denote the number of connected components in $D_2\atop D_1$ whose vertices are all in the middle row. 
Let $D_2\star D_1$ denote a partition diagram of type $k\to m$ with the following property: vertices are in the same connected component of $D_2\star D_1$ if and only if the corresponding vertices in the top and bottom rows of $D_2\atop D_1$ are in the same connected component. Note that $\beta(D_1,D_2)$ and the equivalence class of $D_2\star D_1$ are well-defined independent of the choice of equivalence class representatives for $D_1$ and $D_2$.  For example, if
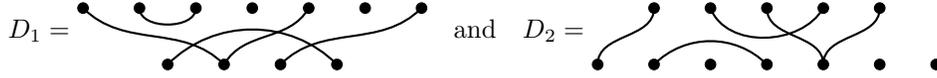
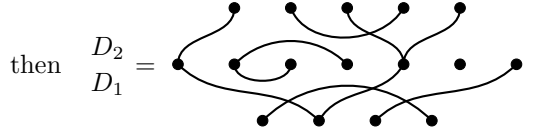
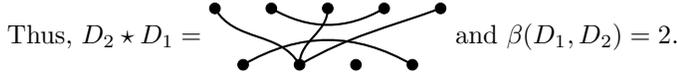
\[
    D_1=
    \begin{tikzpicture}[baseline = 10pt, scale=0.75]
        %Draw bottom vertices
        \draw \foreach \m in {1,...,4} {
            (\m+1.5,0) node[circle, draw, fill=black]{}
            };
            %Draw top vertices
            \draw \foreach \m in {1,...,7} {
            (\m,1) node[circle, draw, fill=black]{}
            };
            %Draw edges:
            \draw[-,thick] (1,1) to[out=-45,in=135] (3.5,0) to[out=60,in=-120] (5,1);
            \draw[-,thick] (2,1) to[out=down, in=down] (3,1);
            \draw[-,thick] (2.5,0) to[out=45, in=135] (5.5,0);
            \draw[-,thick] (4.5,0) to[out=45,in=-135] (7,1);
    \end{tikzpicture}
    \quad\text{and}\quad
    D_2=
    \begin{tikzpicture}[baseline = 10pt, scale=0.75]
        %Draw bottom vertices
        \draw \foreach \m in {1,...,7} {
            (\m,0) node[circle, draw, fill=black]{}
            };
            %Draw top vertices
            \draw \foreach \m in {1,...,5} {
            (\m+1,1) node[circle, draw, fill=black]{}
            };
            %Draw edges:
            \draw[-,thick] (2,1) to[out=down,in=up] (1,0);
            \draw[-,thick] (2,0) to[out=45, in=135] (4,0);
            \draw[-,thick] (3,1) to[out=-69, in=-120] (5,1);
            \draw[-,thick] (4,1) to[out=down,in=up] (5,0) to[out=up,in=down] (6,1);
    \end{tikzpicture}
\]
\[
    \text{then}\quad
    {D_2\atop D_1}= ~
    \begin{tikzpicture}[baseline = 18pt, scale=0.75]
        %Draw bottom vertices
        \draw \foreach \m in {1,...,4} {
            (\m+1.5,0) node[circle, draw, fill=black]{}
            };
            %Draw middle vertices
            \draw \foreach \m in {1,...,7} {
            (\m,1) node[circle, draw, fill=black]{}
            };
            %Draw top vertices
            \draw \foreach \m in {1,...,5} {
            (\m+1,2) node[circle, draw, fill=black]{}
            };
            %Draw bottom edges:
            \draw[-,thick] (1,1) to[out=-45,in=135] (3.5,0) to[out=60,in=-120] (5,1);
            \draw[-,thick] (2,1) to[out=down, in=down] (3,1);
            \draw[-,thick] (2.5,0) to[out=45, in=135] (5.5,0);
            \draw[-,thick] (4.5,0) to[out=45,in=-135] (7,1);
            %Draw top edges:
            \draw[-,thick] (2,2) to[out=down,in=up] (1,1);
            \draw[-,thick] (2,1) to[out=45, in=135] (4,1);
            \draw[-,thick] (3,2) to[out=-69, in=-120] (5,2);
            \draw[-,thick] (4,2) to[out=down,in=up] (5,1) to[out=up,in=down] (6,2);
    \end{tikzpicture}\quad.
\]
Thus, 
$D_2\star D_1=\mathord{
\begin{tikzpicture}[baseline = 8pt, scale=0.75]
    %Draw bottom vertices
    \draw \foreach \m in {1,...,4} {
        (\m+0.5,0) node[circle, draw, fill=black]{}
        };
        %Draw top vertices
        \draw \foreach \m in {1,...,5} {
        (\m,1) node[circle, draw, fill=black]{}
        };
        %Draw edges:
        \draw[-,thick] (2,1) to[out=-30, in=-150] (4,1);
        \draw[-,thick] (1.5,0) to[out=30, in=150] (4.5,0);
        \draw[-,thick] (1,1) to[out=-60,in=120] (2.5,0);
        \draw[-,thick] (3,1) to[out=down,in=up] (2.5,0);
        \draw[-,thick] (5,1) to[out=-160,in=30] (2.5,0);
\end{tikzpicture}
}$ 
and $\beta(D_1,D_2)=2$. 

\subsection{The category $\P(n)$}\label{subsection: the category P}
Given $n\in\k$ we define the \emph{partition category} $\P(n)$
to be the
category with nonnegative integers as objects and
morphisms $\Hom_{\P(n)}(k, \ell)$ consisting of all formal
$\k$-linear combinations of equivalence classes of
partition diagrams of type $k\to \ell$.
Composition of partition diagrams is defined by setting $D_2 \circ D_1=n^{\beta(D_1,D_2)}D_2\star D_1$; it is easy to check that this is associative.
There is also a well-defined tensor product making $\P(n)$ into a strict 
monoidal category.
This is 
defined on diagrams so that $D_1 \otimes D_2$ is obtained by horizontally stacking $D_1$
to the left of $D_2$.

The endomorphism algebras in $\P(n)$ are the \emph{partition algebras} introduced in \cite{Martin1991}. We will write $P_k(n)=\End_{\P(n)}(k)$. 

The monoidal category $\P(n)$ can also be described via generators and relations. It is easy to show that $\P(n)$ is generated as a monoidal category by the following partition diagrams: 
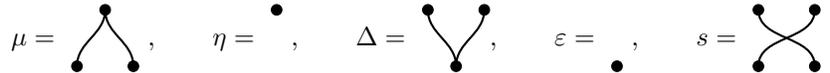
\[\mult=~
    \begin{tikzpicture}[baseline = 8pt, scale=0.75]
         %Draw bottom vertices
        \draw \foreach \m in {1,...,2} {
        (\m,0) node[circle, draw, fill=black]{}
        };%Draw top vertices
        \draw \foreach \m in {1,...,1} {
        (\m+0.5,1) node[circle, draw, fill=black]{}
        };
        %Draw edges:
        \draw[-,thick] (1,0) to[out=up, in=down] (1.5,1) to[out=down, in=up] (2,0);
    \end{tikzpicture}
    ~,\qquad
    \un=~
    \begin{tikzpicture}[baseline = 8pt, scale=0.75]
        %Draw top vertices
        \draw \foreach \m in {1,...,1} {
        (\m,1) node[circle, draw, fill=black]{}
        };
    \end{tikzpicture}
    ~,\qquad\comult=~
    \begin{tikzpicture}[baseline = 8pt, scale=0.75]
         %Draw bottom vertices
        \draw \foreach \m in {1,...,1} {
        (\m+0.5,0) node[circle, draw, fill=black]{}
        };%Draw top vertices
        \draw \foreach \m in {1,...,2} {
        (\m,1) node[circle, draw, fill=black]{}
        };
        %Draw edges:
        \draw[-,thick] (1,1) to[out=down, in=up] (1.5,0) to[out=up, in=down] (2,1);
    \end{tikzpicture}
    ,\qquad\coun=~
    \begin{tikzpicture}[baseline = 8pt, scale=0.75]
        %Draw bottom vertices
        \draw \foreach \m in {1,...,1} {
        (\m,0) node[circle, draw, fill=black]{}
        };
    \end{tikzpicture}
    ~,\qquad\symm=~
    \begin{tikzpicture}[baseline = 8pt, scale=0.75]
         %Draw bottom vertices
        \draw \foreach \m in {1,...,2} {
        (\m,0) node[circle, draw, fill=black]{}
        };%Draw top vertices
        \draw \foreach \m in {1,...,2} {
        (\m,1) node[circle, draw, fill=black]{}
        };
        %Draw edges:
        \draw[-,thick] (1,0) to[out=up,in=down] (2,1);
        \draw[-,thick] (2,0) to[out=up,in=down] (1,1);
    \end{tikzpicture}
~.\]
The morphism $s$ prescribes a symmetric braiding on $\P(n)$ and $(1,\mult,\un,\comult,\coun)$ is a special commutative Frobenius algebra. In fact, $\P(n)$ is equivalent to the free $\k$-linear symmetric monoidal category generated by an $n$-dimensional special commutative Frobenius algebra. Translating this into a statement concerning generators and relations gives us the content of the following theorem, which is essentially a consequence of a result of Abrams \cite{Abrams} on the category of 2-dimensional cobordisms. We refer the interested reader to \cite{Kock} for more details on the connection between 2-dimensional cobordisms and commutative Frobenius algebras.  

\begin{theorem}\label{g and r for P}
As a $\k$-linear monoidal category, $\P(n)$ is generated by the object $1$ and the morphisms $\mult:2\to 1, \un:0\to 1, \comult:1\to 2, \coun:1\to 0, \symm:2\to 2$ 
subject only to the following relations:
\begin{align}
\label{S}
\symm^2&=1_2,\\
\label{B}
(1_1\otimes \symm)\circ(\symm\otimes 1_1)\circ(1_1\otimes \symm)&=(\symm\otimes 1_1)\circ(1_1\otimes \symm)\circ (\symm\otimes 1_1),\\
\symm\circ(1_1\otimes \un)&=\un\otimes 1_1, \\
(1_1\otimes \mult)\circ(\symm\otimes 1_1)\circ(1_1\otimes \symm)&=\symm\circ(\mult\otimes 1_1),\\ 
(1_1\otimes \coun)\circ \symm&=\coun\otimes 1_1, \\
(1_1\otimes \symm)\circ(\symm\otimes 1_1)\circ(1_1\otimes \comult)&=(\comult\otimes 1_1)\circ \symm,\\
\label{unit}
\mult\circ (1_1\otimes \un)&=1_1,\\% \mult\circ(\un\otimes 1_1)=1_1 follows from commutative
\label{counit}
(1_1\otimes \coun)\circ \comult&=1_1=(\coun\otimes1_1)\circ\comult,\\
\label{frobenious}
(\mult\otimes 1_1)\circ (1_1\otimes \comult)&=\comult\circ \mult=(1_1\otimes \mult)\circ(\comult\otimes 1_1),\\
\label{commutative} 
\mult\circ \symm&=\mult,\\
\mult\circ \comult &= 1_1,\\
\label{last P relation}
\coun\circ \un &= n.
 \end{align}
\end{theorem}

\begin{proof}
    By \cite{Abrams}, the category $2Cob$ of 2-dimensional cobordisms is equivalent to the free symmetric monoidal category generated by a commutative Frobenius algebra $(A,\mult_A,\un_A,\comult_A,\coun_A)$. Let $2Cob_\k$ denote the $\k$-linearization of $2Cob$ (i.e. the category with the same objects but with morphisms given by formal $\k$-linear combinations of morphisms in $2Cob$).  It is easy to see that $\P(n)$ is equivalent to the category obtained from $2Cob_\k$ by factoring out by the relations $\mult_A\circ\comult_A=1_A$ (the Frobenius algebra is special) and $\coun_A\circ\un_A=n$ ($A$ has dimension $n$). Thus $\P(n)$ is equivalent to the free symmetric monoidal category generated by an $n$-dimensional special commutative Frobenius algebra. The result can now be deduced from the presentation of $2Cob$ found in \cite{Kock}. 
\end{proof}

\subsection{Tensor powers of the natural representation of $S_n$}\label{repn Sn}

Let $V=\k^n$ with standard basis of unit vectors $v_1,\ldots,v_n$. The symmetric group $S_n$ acts on $V$ by permuting those basis elements: $\sigma\cdot v_i=v_{\sigma(i)}$ for all $\sigma\in S_n$ and $1\leq i\leq n$. For each tuple ${\bd i}=(i_1,\ldots,i_k)$ we set $v_{\bd i}:=v_{i_1}\otimes\cdots\otimes v_{i_k}$. The set $\{v_{\bd i} : 1\leq i_1,\ldots,i_k\leq n\}$ is a basis for the tensor power $V^{\otimes k}$. By convention we set $v_{()}=1$ in $V^{\otimes 0}=\unit$ (the trivial representation). The induced action of the symmetric group $S_n$ on $V^{\otimes k}$ also permutes our chosen basis: $\sigma\cdot v_{\bf i}=v_{\sigma(i_1)}\otimes\cdots\otimes v_{\sigma(i_k)}$. 

Given a partition diagram $D:k\to0$ we let $O_D$ denote the set of all $v_{\bd i}\in V^{\otimes k}$ such that $i_\ell=i_m$ if and only if the $\ell$th and $m$th vertices of $D$ are in the same part. For example, if 
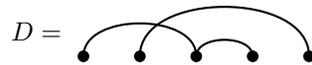
\[
    D=~
    \begin{tikzpicture}[baseline = 6pt, scale=0.75]
        %Draw bottom vertices
        \draw \foreach \m in {1,...,5} {
        (\m,0) node[circle, draw, fill=black]{}
        };
        %Draw edges:
        \draw[-,thick] (1,0) to[out=up, in=up] (3,0) to[out=up, in=up] (4,0);
        \draw[-,thick] (2,0) to[out=up, in=up] (5,0);
    \end{tikzpicture}
\]
then $O_D=\{v_i\otimes v_\ell\otimes v_i\otimes v_i\otimes v_\ell : 1\leq i\not=\ell\leq n\}$. 
Given a partition diagram $D:k\to 0$, let $f_D:V^{\otimes k}\to\unit$ be the $\k$-linear map defined by  
\[
    f_{D}(v_{\bd i})=
    \begin{cases}   
        1, & \text{if }v_{\bd i}\in O_D;\\
        0, & \text{if }v_{\bd i}\not\in O_D.
    \end{cases}
\]
It is easy to see that the $S_n$-orbits of our chosen basis for $V^{\otimes k}$ are precisely the $O_D$ with $D:k\to 0$ having at most $n$ parts. As a consequence, we have the following:

\begin{proposition}\label{prop: f basis}
The set of all $f_D$ with $D:k\to 0$ having at most $n$ parts is a basis for the space $\Hom_{S_n}(V^{\otimes k},\unit)$.
\end{proposition}

\subsection{The functor $\Phi$}\label{subsection: the functor Phi}

There is a monoidal functor $\Phi:\P(n)\to\Rep(S_n)$ defined on generators by setting $\Phi(1)=V$ and 
\[\begin{array}{rclcrclcrcl}
    \Phi(\mult): V\otimes V & \to & V & & v_i\otimes v_k & \mapsto & \delta_{i,k} v_i\\
    \Phi(\un): \unit & \to & V & & 1 & \mapsto & \sum_{i=1}^n v_i\\
    \Phi(\comult): V & \to & V\otimes V & & v_i & \mapsto & v_i \otimes v_i  \\
    \Phi(\coun): V & \to & \unit & & v_i & \mapsto & 1 \\
    \Phi(\symm): V\otimes V & \to & V\otimes V & & u\otimes w & \mapsto & w\otimes u
\end{array}\]
Indeed, it is easy to check that the linear maps above satisfy the corresponding relations in Theorem \ref{g and r for P} and commute with the action of $S_n$. 

Alternatively, $\Phi$ can be described as follows: given a partition diagram $D:k\to\ell$, the $\k$-linear map $\Phi(D):V^{\otimes k}\to V^{\otimes\ell}$ sends $v_{\bd i}$ to the sum of all $v_{\bd i'}$ such that the labeling of the bottom and top vertices of $D$ by the entries of ${\bd i}$ and ${\bd i'}$ respectively has the property the labels of any two vertices in the same part are equal. For example, the vertices of the following partition diagram are labeled by arbitrary tuples ${\bd i}$ and ${\bd i'}$:
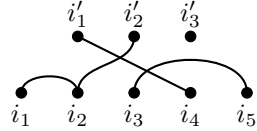
\[
    \begin{tikzpicture}[baseline = 6pt, scale=0.75]
        %Draw vertices
        \draw \foreach \m in {1,...,5} {
        (\m,0) node[circle, draw, fill=black]{}
        };
        \draw \foreach \m in {2,...,4} {
        (\m,1) node[circle, draw, fill=black]{}
        };
        %Draw edges:
        \draw[-,thick] (1,0) to[out=up, in=up] (2,0) to[out=up, in=down] (3,1);
        \draw[-,thick] (3,0) to[out=up, in=up] (5,0);
        \draw[-,thick] (4,0) to (2,1);
        %Labels
        \draw \foreach \m in {1,...,5} {
        (\m,-0.4) node{$i_\m$}
        };
        \draw \foreach \m in {1,...,3} {
        (\m+1,1.4) node{$i'_\m$}
        };
    \end{tikzpicture}
\]
The corresponding map $V^{\otimes 5}\to V^{\otimes 3}$ sends $v_{\bd i}\mapsto\delta_{i_1,i_2}\delta_{i_3,i_5}\sum\limits_{1\leq m\leq n}v_{i_4}\otimes v_{i_1}\otimes v_m$.
In particular, given a partition diagram $D:k\to 0$ we have 
\begin{equation}\label{Phi as sum of f}
\Phi(D)=\sum_{D'\geq D}f_{D'}.
\end{equation}
Motivated by the equation above, we recursively define a new basis $\{x_D\}_{D:k\to\ell}$ for $\Hom_{\P(n)}(k,\ell)$ by setting 
\begin{equation}
x_D=D-\sum_{D'\gneqq D}x_{D'}.
\end{equation}
For example, if 
\[
    D=~
    \begin{tikzpicture}[baseline = 8pt, scale=0.75]
        %Draw vertices
        \draw \foreach \m in {1,2} {
        (\m,0) node[circle, draw, fill=black]{}
        (\m,1) node[circle, draw, fill=black]{}
        };
        %Draw edges:
        \draw[-,thick] (1,0) to[out=up, in=down] (2,1);
    \end{tikzpicture}
    \quad\text{then}\quad
    x_D=~
    \begin{tikzpicture}[baseline = 8pt, scale=0.75]
        %Draw vertices
        \draw \foreach \m in {1,2} {
        (\m,0) node[circle, draw, fill=black]{}
        (\m,1) node[circle, draw, fill=black]{}
        };
        %Draw edges:
        \draw[-,thick] (1,0) to[out=up, in=down] (2,1);
    \end{tikzpicture}
    ~-~
    \begin{tikzpicture}[baseline = 8pt, scale=0.75]
        %Draw vertices
        \draw \foreach \m in {1,2} {
        (\m,0) node[circle, draw, fill=black]{}
        (\m,1) node[circle, draw, fill=black]{}
        };
        %Draw edges:
        \draw[-,thick] (1,0) to[out=up, in=down] (2,1);
        \draw[-,thick] (2,0) to[out=up, in=down] (1,1);
    \end{tikzpicture}
    ~-~
    \begin{tikzpicture}[baseline = 8pt, scale=0.75]
        %Draw vertices
        \draw \foreach \m in {1,2} {
        (\m,0) node[circle, draw, fill=black]{}
        (\m,1) node[circle, draw, fill=black]{}
        };
        %Draw edges:
        \draw[-,thick] (1,0) to[out=up, in=down] (2,1);
        \draw[-,thick] (2,0) to[out=up, in=down] (2,1);
    \end{tikzpicture}
    ~-~
    \begin{tikzpicture}[baseline = 8pt, scale=0.75]
        %Draw vertices
        \draw \foreach \m in {1,2} {
        (\m,0) node[circle, draw, fill=black]{}
        (\m,1) node[circle, draw, fill=black]{}
        };
        %Draw edges:
        \draw[-,thick] (1,0) to[out=up, in=down] (2,1);
        \draw[-,thick] (1,0) to[out=up, in=down] (1,1);
    \end{tikzpicture}
    ~+2~
    \begin{tikzpicture}[baseline = 8pt, scale=0.75]
        %Draw vertices
        \draw \foreach \m in {1,2} {
        (\m,0) node[circle, draw, fill=black]{}
        (\m,1) node[circle, draw, fill=black]{}
        };
        %Draw edges:
        \draw[-,thick] (1,0) to[out=up, in=down] (2,1);
        %\draw[-,thick] (2,0) to[out=up, in=down] (1,1);
        \draw[-,thick] (2,0) to (2,1);
        \draw[-,thick] (1,1) to (1,0);
    \end{tikzpicture}.
\]
Now given any $D:k\to0$, an easy induction argument using (\ref{Phi as sum of f}) shows 
\begin{equation}\label{x to f}
    \Phi(x_D)=f_D.
\end{equation}
In particular, $\Phi(x_D)=0$ whenever $D:k\to0$ has more than $n$ parts. 

\begin{theorem}\label{thm: Phi is full and kernel}
    (Compare with \cite[Theorem 2.6]{CO11})

    (1) The functor $\Phi:\P(n)\to\Rep(S_n)$ is full.

    (2) The kernel of the map 
    \begin{equation}\label{Phi on homs}
    \Phi:\Hom_{\P(n)}(k,\ell)\to\Hom_{S_n}(V^{\otimes k},V^{\otimes \ell})
    \end{equation}
    is the span of all $x_D$ with $D:k\to \ell$ having more than $n$ parts. In particular, (\ref{Phi on homs}) is an isomorphism if and only if $k+\ell\leq n$. 
\end{theorem}

\begin{proof}
    First we reduce the argument to studying (\ref{Phi on homs}) in the case $\ell=0$.  
    % For part (1) we need to show (\ref{Phi on homs}) is surjective for all $k,\ell\geq0$. 
    Toward that end, notice there is a $\k$-linear isomorphism $\Hom_{\P(n)}(k,\ell)\to \Hom_{\P(n)}(k+\ell,0)$ given on diagrams by 
    \begin{equation}\label{cap it}
        \begin{tikzpicture}[baseline = 12pt]
            %Draw bottom vertices
            \draw \foreach \m in {0,1} {
                (\m,0) node[circle, draw, fill=black,inner sep=0pt, minimum width=4pt]{}
                (\m,1.5) node[circle, draw, fill=black,inner sep=0pt, minimum width=4pt]{}
                };
            \draw[-,thick] (0,0) to (0,0.2);
            \draw (0.5,0) node{$\cdots$};
            \draw[-,thick] (1,0) to (1,0.2);
            \draw[-,thick] (0,1.5) to (0,0.8);
            \draw (0.5,1.25) node{$\cdots$};
            \draw[-,thick] (1,1.5) to (1,0.8);
            \draw[-] (-0.2,0.2) to (1.2,0.2) to (1.2,0.8) to (-0.2,0.8) to (-0.2,0.2);
            \draw (0.5,0.5) node{?};
        \end{tikzpicture}
        ~\mapsto~
        \begin{tikzpicture}[baseline = 12pt]
            %Draw bottom vertices
            \draw \foreach \m in {0,1,1.5,2.5} {
                (\m,0) node[circle, draw, fill=black,inner sep=0pt, minimum width=4pt]{}
                };
            \draw[-,thick] (0,0) to (0,0.2);
            \draw (0.5,0) node{$\cdots$};
            \draw[-,thick] (1,0) to (1,0.2);
            \draw[-,thick] (0,0.8) to[out=up,in=up] (2.5,0.8) to[out=down,in=up] (2.5,0);
            \draw (2,0) node{$\cdots$};
            \draw[-,thick] (1,0.8) to[out=up,in=up] (1.5,0.8) to[out=down,in=up] (1.5,0);
            \draw[-] (-0.2,0.2) to (1.2,0.2) to (1.2,0.8) to (-0.2,0.8) to (-0.2,0.2);
            \draw (0.5,0.5) node{?};
        \end{tikzpicture}    
    \end{equation}
    with inverse prescribed by 
    \begin{equation}\label{uncap it}
        \begin{tikzpicture}[baseline = 12pt]
            %Draw bottom vertices
            \draw \foreach \m in {0,1,1.5,2.5} {
                (\m,-0.5) node[circle, draw, fill=black,inner sep=0pt, minimum width=4pt]{}
                };
            \draw[-,thick] (0,-0.5) to (0,0.2);
            \draw (0.5,-0.25) node{$\cdots$};
            \draw[-,thick] (1,-0.5) to (1,0.2);
            \draw[-,thick] (1.5,-0.5) to (1.5,0.2);
            \draw (2,-0.25) node{$\cdots$};
            \draw[-,thick] (2.5,-0.5) to (2.5,0.2);
            \draw[-] (-0.2,0.2) to (2.7,0.2) to (2.7,0.8) to (-0.2,0.8) to (-0.2,0.2);
            \draw (1.25,0.5) node{?};
        \end{tikzpicture}
        ~\mapsto~
        \begin{tikzpicture}[baseline = 12pt]
            %Draw bottom vertices
            \draw \foreach \m in {0,1} {
                (\m,-0.5) node[circle, draw, fill=black,inner sep=0pt, minimum width=4pt]{}
                (\m+3,1) node[circle, draw, fill=black,inner sep=0pt, minimum width=4pt]{}
                };
                \draw[-,thick] (0,-0.5) to (0,0.2);
                \draw (0.5,-0.25) node{$\cdots$};
                \draw[-,thick] (1,-0.5) to (1,0.2);
                \draw (3.5,1) node{$\cdots$};
                \draw[-,thick] (2.5,0.2) to[out=down,in=down] (3,0.2) to[out=up,in=down] (3,1);
            \draw[-,thick] (1.5,0.2) to[out=down,in=down] (4,0.2) to[out=up,in=down] (4,1);
                \draw[-] (-0.2,0.2) to (2.7,0.2) to (2.7,0.8) to (-0.2,0.8) to (-0.2,0.2);
                \draw (1.25,0.5) node{?};
        \end{tikzpicture}    
    \end{equation}
    Note that this isomorphism preserves the number of parts in a diagram as well as the partial order $\leq$ on partition diagrams. Hence, for each partition diagram $D:k\to \ell$ the isomorphism maps $x_D\mapsto x_{D'}$ for some partition diagram $D':k+\ell\to0$ with the same number of parts as $D$.

    Similarly, there is a $\k$-linear isomorphism $\Hom_{S_n}(V^{\otimes k}, V^{\otimes \ell})\to\Hom_{S_n}(V^{\otimes k+\ell}, \unit)$ given by 
    $$g\mapsto \Phi(\Cap_\ell)\circ(g\otimes 1_{V^{\otimes\ell}})$$
    with inverse 
    $$g\mapsto (g\otimes 1_{V^{\otimes\ell}})\circ(1_{V^{\otimes k}}\otimes \Phi(\Cup_\ell))$$
    where 
    \[
        \Cap_\ell=~
        \begin{tikzpicture}[baseline = 8pt]
            %Draw bottom vertices
            \draw \foreach \m in {0,1,2,3} {
                (\m,0) node[circle, draw, fill=black,inner sep=0pt, minimum width=4pt]{}
                };
            \draw (0.5,0) node{$\cdots$};
            \draw[-,thick] (0,0) to[out=up,in=left] (1.5,1) to[out=right,in=up] (3,0);
            \draw (2.5,0) node{$\cdots$};
            \draw[-,thick] (1,0) to[out=up,in=left] (1.5,0.75) to[out=right,in=up] (2,0);
            \draw (0,-0.3) node{1};
            \draw (1,-0.3) node{$\ell$};
            \draw (2,-0.3) node{$\ell+1$};
            \draw (3,-0.3) node{$2\ell$};
        \end{tikzpicture}    
        \qquad
        \Cup_\ell=~
        \begin{tikzpicture}[baseline = 16pt]
            %Draw bottom vertices
            \draw \foreach \m in {0,1,2,3} {
                (\m,1) node[circle, draw, fill=black,inner sep=0pt, minimum width=4pt]{}
                };
            \draw (0.5,1) node{$\cdots$};
            \draw[-,thick] (0,1) to[out=down,in=left] (1.5,0) to[out=right,in=down] (3,1);
            \draw (2.5,1) node{$\cdots$};
            \draw[-,thick] (1,1) to[out=down,in=left] (1.5,0.25) to[out=right,in=down] (2,1);
            \draw (0,1.3) node{1};
            \draw (1,1.3) node{$\ell$};
            \draw (2,1.3) node{$\ell+1$};
            \draw (3,1.3) node{$2\ell$};
        \end{tikzpicture}    
    \]
    Since $\Phi$ is a monoidal functor the following diagram commutes:
    \begin{equation}\label{Reduce to invariants symm}
        \begin{CD}
        \Hom_{\P(n)}(k,\ell)&@>>>&
        \Hom_{\P(n)}(k+\ell,0)\\
        @V \Phi VV&
        &@V \Phi VV\\
        \Hom_{S_n}(V^{\otimes k}, V^{\otimes \ell})&@>>>&
        \Hom_{S_n}(V^{\otimes k+\ell}, \unit).
        \end{CD}
    \end{equation}
    Now, the vertical map on the right is surjective and its kernel is the span of all $x_D$ with $D:k+\ell\to 0$ having more than $n$ parts thanks to Proposition \ref{prop: f basis} and (\ref{x to f}). The result follows.
\end{proof}

\section{The jellyfish partition category}\label{section: JP}
In this section we define the jellyfish partition category $\JP(n)$ in terms of generators and relations, and develop a diagram calculus for $\JP(n)$ in terms of jellyfish partition diagrams. First, however, we point out properties of the determinant map which serve as motivation for the definition of $\JP(n)$.

\subsection{Motivation: the determinant map}\label{subsection: Motivation: the determinant map} The determinant can be viewed as a linear map  
\begin{equation*}
\det:V^{\otimes n}\to\k,\quad v_{i_1}\otimes\cdots\otimes v_{i_n}\mapsto \det[v_{i_1}~\cdots~v_{i_n}].
\end{equation*}
An easy, but crucial observation is that for any $\sigma\in S_n$ we have 
\begin{equation}\label{det property}
\det(\sigma\cdot v_{\bd i})=(-1)^\sigma\det(v_{\bd i}).
\end{equation} 
It follows that $\det\in\Hom_{A_n}(V^{\otimes n},\unit)$ but $\det\not\in\Hom_{S_n}(V^{\otimes n},\unit)$. 
Thus, in order to construct an $A_n$-analog of the partition category we must include an analog of the determinant map into our diagrams. This is precisely what the jellyfish in the forthcoming jellyfish partition diagrams are meant to do. On the other hand, it follows from (\ref{det property}) that $\det\otimes\det\in\Hom_{S_n}(V^{\otimes 2n},\unit)$. Thus, it follows from Theorem \ref{thm: Phi is full and kernel} that $\det\otimes\det$ is in the image of $\Phi$. In order to describe a morphism in $\P(n)$ which is mapped to $\det\otimes\det$ by $\Phi$ we will make use of a few special partition diagrams. First, let 
$\crosscap_n:2n\to 0$ 
 denote the following diagram:
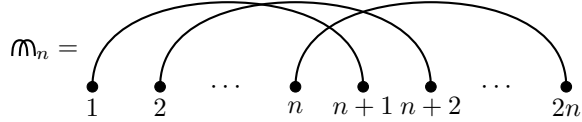
\[
    \crosscap_n=
    \begin{tikzpicture}[baseline = 12pt, scale=0.9]
        \draw (1,0) node[circle, draw, fill=black]{};
        \draw (2,0) node[circle, draw, fill=black]{};
        \draw (3,0) node{$\cdots$};
        \draw (4,0) node[circle, draw, fill=black]{};
        \draw (5,0) node[circle, draw, fill=black]{};
        \draw (6,0) node[circle, draw, fill=black]{};
        \draw (7,0) node{$\cdots$};
        \draw (8,0) node[circle, draw, fill=black]{};
        \draw (1,-0.3) node{1};
        \draw (2,-0.3) node{2};
        \draw (4,-0.3) node{$n$};
        \draw (5,-0.3) node{$n+1$};
        \draw (6,-0.3) node{$n+2$};
        \draw (8,-0.3) node{$2n$};
        \draw[-,thick] (1,0) to[out=up, in=left] (3,1.25) to[out=right, in=up] (5,0);
        \draw[-,thick] (2,0) to[out=up, in=left] (4,1.25) to[out=right, in=up] (6,0);
        \draw[-,thick] (4,0) to[out=up, in=left] (6,1.25) to[out=right, in=up] (8,0);
    \end{tikzpicture}
\]
Also, whenever $1\leq i<k$ we let $\symm_i^{(k)}:k\to k$ denote the following:
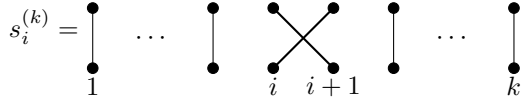
\[
    \symm_i^{(k)}=
    \begin{tikzpicture}[baseline = 12pt, scale=0.8]
        \draw \foreach \m in {1,3,6,8} {
            (\m,0) node[circle, draw, fill=black,inner sep=0pt, minimum width=4pt]{}
            (\m,1) node[circle, draw, fill=black,inner sep=0pt, minimum width=4pt]{}
            (\m,0) to[-,thick] (\m,1)
            };  
        \draw (4,0)node[circle, draw, fill=black,inner sep=0pt, minimum width=4pt]{};
        \draw (4,1)node[circle, draw, fill=black,inner sep=0pt, minimum width=4pt]{};
        \draw (5,0)node[circle, draw, fill=black,inner sep=0pt, minimum width=4pt]{};
        \draw (5,1)node[circle, draw, fill=black,inner sep=0pt, minimum width=4pt]{};
        \draw[-,thick] (4,0) to (5,1);
        \draw[-,thick] (4,1) to (5,0);
        \draw (2,0.5) node{$\cdots$};
        \draw (7,0.5) node{$\cdots$};
        \draw (1,-0.3) node{1};
        \draw (4,-0.3) node{$i$};
        \draw (5,-0.3) node{$i+1$};
        \draw (8,-0.3) node{$k$};
    \end{tikzpicture}
\]
We simply write $\symm_i=\symm_i^{(k)}$ if $k$ is understood from the context. Note that for each $k\in\Z_{>0}$ the elements $\symm_1,\ldots,\symm_{k-1}$ generate a copy of the symmetric group $S_k\subseteq P_k(n)$. We refer to the corresponding partition diagrams as \emph{permutation diagrams}.

The following proposition serves as the motivation relations (\ref{jellyfish symmetry})--(\ref{jellyfish relation}) in the forthcoming definition of the jellyfish partition category. It also shows that the role of $\det\otimes\det$ in $\P(n)$ is played by the following:
\[
    \sum_{\sigma\in S_n}(-1)^\sigma~
    \begin{tikzpicture}[baseline = 12pt]
        %Draw bottom vertices
        \draw[-,thick]  \foreach \m in {1,2,2.5, 3.5} {
            (\m,0) node[circle, draw, fill=black,inner sep=0pt, minimum width=4pt]{}
            (\m, 0) to (\m, 0.25)
            };
        %Draw edges
        \draw[-,thick] (0.9,0.25) to (2.1,0.25) to (2.1,0.75) to (0.9,0.75) to (0.9,0.25);      
        \draw[-,thick] (1,0.75) to[out=up, in=left] (1.75,1.25) to[out=right, in=up] (2.5,0.75) to (2.5,0.25);
        \draw[-,thick] (2,0.75) to[out=up, in=left] (2.75,1.25) to[out=right, in=up] (3.5,0.75) to (3.5,0.25);
        \draw (1.5,.1) node{$\cdots$};
        \draw (1.5,.5) node{$\sigma$};
        \draw (3,.5) node{$\cdots$};  
    \end{tikzpicture}
\]

\begin{proposition}\label{relations for det} 
(1) $\det\otimes 1_V=(1_V\otimes \det)\circ \Phi(s_1\circ\cdots\circ s_n)$.

(2) $\det\circ \Phi(s_i)=-\det$ for all $1\leq i<n$.

(3) $\det\otimes \det = \Phi\left(\sum_{\sigma\in S_n}(-1)^\sigma \crosscap_n\circ(\sigma\otimes 1_n)\right)$.
\end{proposition}

\begin{proof}
Part (1) follows from the fact that $\Phi(s)$ induces the usual symmetric braiding on the category of finite dimensional vector spaces. Part (2) is true because $\det$ is a skew symmetric form. For part (3), first notice $\crosscap_n\circ(\sigma\otimes 1_n)=\crosscap_n\circ(1_n\otimes\sigma^{-1})$ for any permutation diagram $\sigma:n\to n$, as illustrated below:
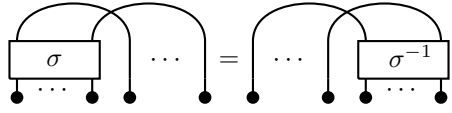
\[
    \begin{tikzpicture}[baseline = 12pt]
        %Draw bottom vertices
        \draw[-,thick]  \foreach \m in {1,2,2.5, 3.5} {
            (\m,0) node[circle, draw, fill=black,inner sep=0pt, minimum width=4pt]{}
            (\m, 0) to (\m, 0.25)
            };
        %Draw edges
        \draw[-,thick] (0.9,0.25) to (2.1,0.25) to (2.1,0.75) to (0.9,0.75) to (0.9,0.25);      
        \draw[-,thick] (1,0.75) to[out=up, in=left] (1.75,1.25) to[out=right, in=up] (2.5,0.75) to (2.5,0.25);
        \draw[-,thick] (2,0.75) to[out=up, in=left] (2.75,1.25) to[out=right, in=up] (3.5,0.75) to (3.5,0.25);
        \draw (1.5,.1) node{$\cdots$};
        \draw (1.5,.5) node{$\sigma$};
        \draw (3,.5) node{$\cdots$};  
    \end{tikzpicture}
    =
    \begin{tikzpicture}[baseline = 12pt]
        %Draw bottom vertices
        \draw[-,thick]  \foreach \m in {1,2,2.5, 3.5} {
            (\m,0) node[circle, draw, fill=black,inner sep=0pt, minimum width=4pt]{}
            (\m, 0) to (\m, 0.25)
            };
        %Draw edges
        \draw[-,thick] (2.4,0.25) to (3.6,0.25) to (3.6,0.75) to (2.4,0.75) to (2.4,0.25);      
        \draw[-,thick] (1,0.25) to (1,0.75) to[out=up, in=left] (1.75,1.25) to[out=right, in=up] (2.5,0.75);
        \draw[-,thick] (2,0.25) to (2,0.75) to[out=up, in=left] (2.75,1.25) to[out=right, in=up] (3.5,0.75);
        \draw (3,.1) node{$\cdots$};
        \draw (3.1,.55) node{$\sigma^{-1}$};
        \draw (1.5,.5) node{$\cdots$};  
    \end{tikzpicture}
\]
It follows that $\sum_{\sigma\in S_n}(-1)^\sigma \crosscap_n\circ(\sigma\otimes 1_n)=\sum_{\sigma\in S_n}(-1)^\sigma \crosscap_n\circ(1_n\otimes\sigma)$. In particular, $\Phi\left(\sum_{\sigma\in S_n}(-1)^\sigma \crosscap_n\circ(\sigma\otimes 1_n)\right):V^{\otimes n}\otimes V^{\otimes n}\to\unit$ is skew symmetric on the left and right $n$ tensors respectively. Since the same is true for $\det\otimes \det$, it suffices to observe that both map $v_1\otimes\cdots \otimes v_n\otimes v_1\otimes\cdots\otimes v_n\mapsto 1$. 
\end{proof}

\subsection{Definition of $\JP(n)$}\label{definition of JP}

We define the \emph{jellyfish partition category} $\JP(n)$ to be the free $\k$-linear monoidal category generated by a single object $1$ and six morphisms $\mult:2\to 1, \un:0\to 1, \comult:1\to 2, \coun:1\to 0, \symm:2\to 2$, and $\jelly:n\to 0$ subject to relations (\ref{S})--(\ref{last P relation}) and the following:
\begin{align}
\label{jellyfish symmetry}
\jelly\otimes 1_1&=(1_1\otimes \jelly)\circ s_1\circ\cdots\circ s_n,\\
\label{skew form relation}
\jelly\circ s_i&=-\jelly\qquad (1\leq i<n),\\
\label{jellyfish relation}
\jelly\otimes \jelly&=\sum_{\sigma\in S_n}(-1)^\sigma \crosscap_n\circ(\sigma\otimes 1_n).
 \end{align}
 The \emph{jellyfish partition algebra} is defined to be $JP_k(n)=\End_{\JP(n)}(k)$.

\begin{remark}
More briefly, $\JP(n)$ is the free $\k$-linear symmetric monoidal category generated by an $n$-dimensional special commutative Frobenius algebra which admits a skew-symmetric $n$-form $j$ satisfying the \emph{jellyfish relation} (\ref{jellyfish relation}).
\end{remark}

By Theorem \ref{g and r for P} there is a monoidal functor $\P(n)\to\JP(n)$ mapping the generators of $\P(n)$ to those of the same name in $\JP(n)$.
Let $\Psi:\JP(n)\to\Rep(A_n)$ denote the monoidal functor which agrees with $\Phi$ on all generators of the same name and with $\Psi(\jelly)=\det$. It follows from Theorem \ref{g and r for P} and Proposition \ref{relations for det} that $\Psi$ is well-defined. Moreover, we get that (\ref{square of functors}) is a commutative diagram of monoidal functors. We will have more to say about the functor $\Psi$ in \S\ref{Psi} and \S\ref{section: Psi is faithful}. First we will develop a diagrammatic description of $\JP(n)$.

\subsection{Jellyfish diagrams}\label{JP diagrams}

By Theorem \ref{g and r for P} we have a functor $\P(n)\to\JP(n)$ mapping the generators of $\P(n)$ to those of the same name in $\JP(n)$. In particular, we can interpret any partition diagram as a morphism in $\JP(n)$. We extend this diagrammatic description to all of $\JP(n)$ by drawing $\jelly:n\to 0$ as a jellyfish with $n$ legs. For instance, in $\JB(5)$ we have:
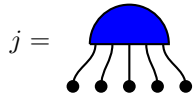
\[
    \jelly=~
    \begin{tikzpicture}[baseline = 20pt, scale=0.75]
        %Draw bottom vertices
        \draw[-,thick]  \foreach \m in {-0.3,0.2,0.7,1.2,1.7} {
            (\m,0.25) node[circle, draw, fill=black]{}};
        %Draw jellyfish
        \filldraw[fill=\jellycolor,-,thick] (0,1) to (1.4,1) to[out=up, in=right] (0.7,1.7) to[out=left, in=up] (0,1);
        %Draw legs
        \draw[-,thick] (-0.3,0.25) to[out=up,in=down] (0.1,1);
        \draw[-,thick] (0.2,0.25) to[out=up,in=down] (0.4,1);
        \draw[-,thick] (0.7,0.25) to[out=up,in=down] (0.7,1);
        \draw[-,thick] (1.2,0.25) to[out=up,in=down] (1,1);
        \draw[-,thick] (1.7,0.25) to[out=up,in=down] (1.3,1);
    \end{tikzpicture}
\]
More generally, a \emph{jellyfish (partition) diagram} refers to any diagram obtained by stacking (horizontally and vertically) any finite number of partition diagrams and $n$-legged jellyfish. Note that all jellyfish in a jellyfish diagram are required to have the same number of legs. 
For example, here is a jellyfish diagram when $n=2$:
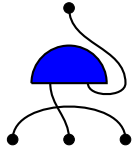
\begin{equation*}\label{small jellyfish diagram}
    \begin{tikzpicture}[baseline = 24pt]
    %Draw bottom vertices
    \draw \foreach \m in {0.25,1,1.75} {
        (\m,0.25) node[circle, draw, fill=black,inner sep=0pt, minimum width=4pt]{}
        };
        %Draw top vertices
        \draw \foreach \m in {1,...,1} {
        (\m,2) node[circle, draw, fill=black,inner sep=0pt, minimum width=4pt]{}
        };
        %Draw jellyfish
        \filldraw[fill=\jellycolor,-,thick] (0.5,1) to (1.5,1) to[out=up, in=right] (1,1.5) to[out=left, in=up] (0.5,1);
        %Draw edges
        \draw[-,thick] (0.25,0.25) to[out=up,in=up] (1.75,0.25);
        \draw[-,thick] (1,0.25) to[out=up,in=down] (0.75,1);
        \draw[-,thick] (1.25,1) to[out=down,in=down] (1.75,1) to[out=up,in=down] (1,2);
    \end{tikzpicture}
\end{equation*}
Every jellyfish diagram should be interpreted as a morphism in $\JP(n)$. For example, the diagram above should be interpreted as the following morphism in $\Hom_{\JP(2)}(3,1)$:
\[(\jelly\otimes 1_1)\circ(1_1\otimes(\comult\circ\coun\circ\un\circ\mult))\circ(\symm\otimes 1_1).\] 
Here is a more complicated jellyfish diagram, which corresponds to a morphism in $\Hom_{\JP(3)}(9,5)$:
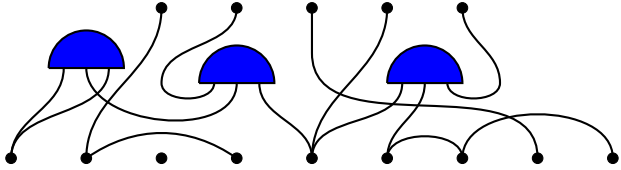
\[\begin{tikzpicture}[baseline = 24pt]
    %Draw bottom vertices
    \draw \foreach \m in {0,...,8} {
        (\m,0) node[circle, draw, fill=black,inner sep=0pt, minimum width=4pt]{}
        };
        %Draw top vertices
        \draw \foreach \m in {2,...,6} {
        (\m,2) node[circle, draw, fill=black,inner sep=0pt, minimum width=4pt]{}
        };
        %Draw jellyfish
        \filldraw[fill=\jellycolor,-,thick] (0.5,1.2) to (1.5,1.2) to[out=up, in=right] (1,1.7) to[out=left, in=up] (0.5,1.2);
        \filldraw[fill=\jellycolor,-,thick] (2.5,1) to (3.5,1) to[out=up, in=right] (3,1.5) to[out=left, in=up] (2.5,1);
        \filldraw[fill=\jellycolor,-,thick] (5,1) to (6,1) to[out=up, in=right] (5.5,1.5) to[out=left, in=up] (5,1);
        %Draw edges
        \draw[-,thick] (2,2) to[out=down,in=up] (1,0);
        \draw[-,thick] (0.7,1.2) to[out=down,in=up] (0,0);
        \draw[-,thick] (1.3,1.2) to[out=down,in=up] (0,0);
        \draw[-,thick] (1,1.2) to[out=down,in=down] (3,1);
        \draw[-,thick] (2.7,1) to[out=down,in=down] (2,1) to[out=up,in=down] (3,2);
        \draw[-,thick] (1,0) to[out=35,in=145] (3,0);
        \draw[-,thick] (6,0) to[out=up,in=up] (8,0);
        \draw[-,thick] (5,0) to[out=up,in=up] (6,0);
        %\draw[-,thick] (3.3,1) to[out=down,in=up] (7,0);
        %\draw[-,thick] (5.2,1) to[out=down,in=down] (3.8,1) to[out=up,in=down] (4,2);
        %\draw[-,thick] (3.3,1) to[out=down,in=down] (5.2,1);
        \draw[-,thick] (3.3,1) to[out=down,in=up] (4,0)to [out=up,in=down] (5.2,1);
        \draw[-,thick] (4,2) to[out=down,in=up] (4,1.4) to[out=down,in=up] (7,0);
        \draw[-,thick] (5.8,1) to[out=down,in=down] (6.5,1) to[out=up,in=down] (6,2);
        \draw[-,thick] (4,0) to[out=up,in=down] (5,2);
         \draw[-,thick] (5.5,1) to[out=down,in=up] (5,0);
\end{tikzpicture}\] 

In the examples of jellyfish diagrams given above we have omitted any ``middle vertices" that arise when stacking partition diagrams and jellyfish so that we are left with vertices in only the top and bottom rows. Moreover, all jellyfish legs have ended at either a top or bottom vertex, or at another jellyfish. Now, attaching the end of a jellyfish leg to any vertex in the same component gives rise to the same morphism in $\JP(n)$:
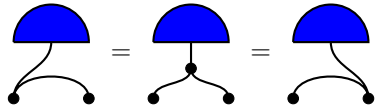
\[
    \begin{tikzpicture}[baseline = 15pt]
        %Draw vertices
        \draw \foreach \m in {0,1} {
            (\m,0) node[circle, draw, fill=black,inner sep=0pt, minimum width=4pt]{}
            };
        %Draw jellyfish
        \filldraw[fill=\jellycolor,-,thick] (0,0.75) to (1,0.75) to[out=up, in=right] (0.5,1.25) to[out=left, in=up] (0,0.75);
        %Draw edges
        \draw[-,thick] (0.5,0.75) to[out=down,in=up] (0,0) to[out=up,in=up] (1,0);
    \end{tikzpicture}
    ~=~
    \begin{tikzpicture}[baseline = 15pt]
        %Draw vertices
        \draw \foreach \m in {0,1} {
            (\m,0) node[circle, draw, fill=black,inner sep=0pt, minimum width=4pt]{}
            };
        %Draw jellyfish
        \filldraw[fill=\jellycolor,-,thick] (0,0.75) to (1,0.75) to[out=up, in=right] (0.5,1.25) to[out=left, in=up] (0,0.75);
        %Draw edges
        \draw[-,thick] (0.5,0.75) to (0.5,0.4) to[out=down,in=up] (0,0);
        \draw[-,thick] (0.5,0.4) to[out=down,in=up] (1,0);
        \draw (0.5,0.4) node[circle, draw, fill=black]{};
    \end{tikzpicture}
    ~=~
    \begin{tikzpicture}[baseline = 15pt]
        %Draw vertices
        \draw \foreach \m in {0,1} {
            (\m,0) node[circle, draw, fill=black,inner sep=0pt, minimum width=4pt]{}
            };
        %Draw jellyfish
        \filldraw[fill=\jellycolor,-,thick] (0,0.75) to (1,0.75) to[out=up, in=right] (0.5,1.25) to[out=left, in=up] (0,0.75);
        %Draw edges
        \draw[-,thick] (0.5,0.75) to[out=down,in=up] (1,0) to[out=up,in=up] (0,0);
    \end{tikzpicture}
\]
Thus, if a jellyfish leg is connected to a vertex $a$, and $a$ connects to another jellyfish or the top or bottom row of vertices, then the diagram can be drawn without vertex $a$.  
In some cases, however, stacking a jellyfish atop a partition diagram will result in a vertex which is connected to a jellyfish leg but not connected to another jellyfish nor any top or bottom vertex. For example, if we stack $\jelly$ atop $(\un\otimes 1_2)$ (which is only allowed if $n=3$) we get
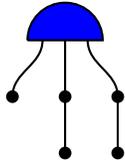
\[
    \begin{tikzpicture}[baseline = 12pt]
        %Draw bottom vertices
        \draw[-,thick]  \foreach \m in {1.5, 2.2} {
            (\m,-1) node[circle, draw, fill=black,inner sep=0pt, minimum width=4pt]{}
            };
        \draw[-,thick]  \foreach \m in {0.8, 1.5, 2.2} {
            (\m,-0.25) node[circle, draw, fill=black,inner sep=0pt, minimum width=4pt]{}
            };
        %Draw jellyfish
        \filldraw[fill=\jellycolor,-,thick] (1,0.5) to (2,0.5) to[out=up, in=right] (1.5,1) to[out=left, in=up] (1,0.5);
        %Draw legs
        \draw[-,thick] (1.5,-0.25) to[out=up,in=down] (1.5,0.5);
        \draw[-,thick] (2.2,-0.25) to[out=up,in=down] (1.8,0.5);
        \draw[-,thick] (0.8,-0.25) to[out=up,in=down] (1.2,0.5);
        \draw[-,thick] (2.2,-1) to (2.2,-0.25);
        \draw[-,thick] (1.5,-1) to (1.5,-0.25);
    \end{tikzpicture}
\]
In these cases we will draw the resulting jellyfish diagram with a \emph{dangling leg}. For example,  we draw the diagram above as follows:
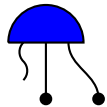
\[
    \begin{tikzpicture}[baseline = 12pt]
        %Draw vertices
        \draw[-,thick]  \foreach \m in {1.5, 2.2} {
            (\m,-0.25) node[circle, draw, fill=black,inner sep=0pt, minimum width=4pt]{}};
        %Draw jellyfish
        \filldraw[fill=\jellycolor,-,thick] (1,0.5) to (2,0.5) to[out=up, in=right] (1.5,1) to[out=left, in=up] (1,0.5);
        %Draw legs
        \draw[-,thick] (1.5,-0.25) to[out=up,in=down] (1.5,0.5);
        \draw[-,thick] (2.2,-0.25) to[out=up,in=down] (1.8,0.5);
        %Draw dangling leg
        \draw[-,thick] (1.2, 0.5) to[out=-135,in=135] (1.2,0.25) to[out=-45,in=45] (1.2,0);
    \end{tikzpicture}
\]
Allowing for dangling legs (which can always be interpreted as occurrences of $\eta$) we can realize any morphism in $\JP(n)$ as a $\k$-linear combination of jellyfish diagrams with whose vertices are only in the top and bottom row. Moreover, all jellyfish legs can be drawn so that they are either dangling, connected to a vertex, or connected to another jellyfish. 

\subsection{Some relations on jellyfish diagrams}\label{JP diagram moves}

In this subsection we will describe some of the consequences of relations (\ref{jellyfish symmetry})--(\ref{jellyfish relation}) in terms of jellyfish diagrams. Relation (\ref{jellyfish symmetry}) implies that jellyfish are allowed to freely swim across other strands:
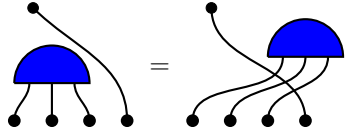
\[
    \begin{tikzpicture}[baseline = 18pt]
        %Draw vertices
        \draw[-,thick]  \foreach \m in {1,1.5,2,2.5} {
            (\m,0) node[circle, draw, fill=black]{}
            };
        \draw (1.25,1.5) node[circle, draw, fill=black]{};
        %Draw jellyfish
        \filldraw[fill=\jellycolor,-,thick] (1,0.5) to (2,0.5) to[out=up, in=right] (1.5,1) to[out=left, in=up] (1,0.5);
        %Draw legs
        \draw[-,thick] (1,0) to[out=up,in=down] (1.2,0.5);
        \draw[-,thick] (1.5,0) to[out=up,in=down] (1.5,0.5);
        \draw[-,thick] (2,0) to[out=up,in=down] (1.8,0.5);
        %Draw edge
        \draw[-,thick] (2.5,0) to[out=up,in=-45] (1.25,1.5);
    \end{tikzpicture}
    ~=~\begin{tikzpicture}[baseline = 18pt]
        %Draw vertices
        \draw[-,thick]  \foreach \m in {1,1.5,2,2.5} {
            (\m,0) node[circle, draw, fill=black]{}
            };
        \draw (1.25,1.5) node[circle, draw, fill=black]{};
        %Draw jellyfish
        \filldraw[fill=\jellycolor,-,thick] (2,0.85) to (3,0.85) to[out=up, in=right] (2.5,1.35) to[out=left, in=up] (2,0.85);
        %Draw legs
        \draw[-,thick] (1,0) to[out=up,in=down] (2.2,0.85);
        \draw[-,thick] (1.5,0) to[out=up,in=down] (2.5,0.85);
        \draw[-,thick] (2,0) to[out=up,in=down] (2.8,0.85);
        %Draw edge
        \draw[-,thick] (2.5,0) to[out=up,in=down] (1.25,1.5);
    \end{tikzpicture}
\] 
In particular, the right and left dual of $\jelly$ are equal:  
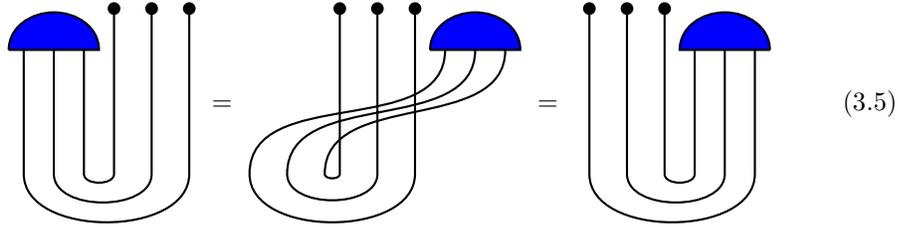
\begin{equation}\label{dual jelly}
    \begin{tikzpicture}[baseline = -10pt]
        %Vertices
        \draw[-,thick] (0,1) node[circle, draw, fill=black]{};
        \draw[-,thick] (0.5,1) node[circle, draw, fill=black]{};
        \draw[-,thick] (1,1) node[circle, draw, fill=black]{};
        %Jellyfish
        \filldraw[fill=\jellycolor,-,thick] (-1.4,0.45) to (-0.2,0.45) to[out=up, in=right] (-0.8,0.95) to[out=left, in=up] (-1.4,0.45);
        %Legs
        \draw[-,thick] (1,1) to[out=down,in=up] (1,-1.2) to[out=down,in=down] (-1.2,-1.2) to[out=up,in=down] (-1.2,0.45);
        \draw[-,thick] (0.5,1) to[out=down,in=up] (0.5,-1.2) to[out=down,in=down] (-0.8,-1.2) to[out=up,in=down] (-0.8,0.45);
        \draw[-,thick] (0,1) to[out=down,in=up] (0,-1.2) to[out=down,in=down] (-0.4,-1.2) to[out=up,in=down] (-0.4,0.45);
    \end{tikzpicture}~=~
    \begin{tikzpicture}[baseline = -10pt]
        %Vertices
        \draw[-,thick] (0,1) node[circle, draw, fill=black]{};
        \draw[-,thick] (0.5,1) node[circle, draw, fill=black]{};
        \draw[-,thick] (1,1) node[circle, draw, fill=black]{};
        %Jellyfish
        \filldraw[fill=\jellycolor,-,thick] (1.2,0.45) to (2.4,0.45) to[out=up, in=right] (1.8,0.95) to[out=left, in=up] (1.2,0.45);
        %Legs
        \draw[-,thick] (1,1) to[out=down,in=up] (1,-1.2) to[out=down,in=down] (-1.2,-1.2) to[out=up,in=down] (1.4,0.45);
        \draw[-,thick] (0.5,1) to[out=down,in=up] (0.5,-1.2) to[out=down,in=down] (-0.7,-1.2) to[out=up,in=down] (1.8,0.45);
        \draw[-,thick] (0,1) to[out=down,in=up] (0,-1.2) to[out=down,in=down] (-0.2,-1.2) to[out=up,in=down] (2.2,0.45);
    \end{tikzpicture}~=~
    \reflectbox{
    \begin{tikzpicture}[baseline = -10pt]
        %Vertices
        \draw[-,thick] (0,1) node[circle, draw, fill=black]{};
        \draw[-,thick] (0.5,1) node[circle, draw, fill=black]{};
        \draw[-,thick] (1,1) node[circle, draw, fill=black]{};
        %Jellyfish
        \filldraw[fill=\jellycolor,-,thick] (-1.4,0.45) to (-0.2,0.45) to[out=up, in=right] (-0.8,0.95) to[out=left, in=up] (-1.4,0.45);
        %Legs
        \draw[-,thick] (1,1) to[out=down,in=up] (1,-1.2) to[out=down,in=down] (-1.2,-1.2) to[out=up,in=down] (-1.2,0.45);
        \draw[-,thick] (0.5,1) to[out=down,in=up] (0.5,-1.2) to[out=down,in=down] (-0.8,-1.2) to[out=up,in=down] (-0.8,0.45);
        \draw[-,thick] (0,1) to[out=down,in=up] (0,-1.2) to[out=down,in=down] (-0.4,-1.2) to[out=up,in=down] (-0.4,0.45);
    \end{tikzpicture}
    }
\end{equation}
As a consequence, we can draw rotated jellyfish in our diagrams without ambiguity. In particular, the following jellyfish diagram should be interpreted to be equal to any of the expressions in (\ref{dual jelly}):
\begin{equation*}\label{j*}
    \jelly^*=~\begin{tikzpicture}[baseline = 12pt]
        %Draw vertices
        \draw[-,thick]  \foreach \m in {1,2} {
            (\m,1) node[circle, draw, fill=black]{}
            };
        \draw (1.5,1) node{$\cdots$};
        %Draw jellyfish
        \filldraw[fill=\jellycolor,-,thick] (1,0.5) to (2,0.5) to[out=down, in=right] (1.5,0) to[out=left, in=down] (1,0.5);
        %Draw legs
        \draw[-,thick] (1,1) to[out=down,in=up] (1.2,0.5);
        \draw[-,thick] (2,1) to[out=down,in=up] (1.8,0.5);
    \end{tikzpicture}
\end{equation*}

By relation (\ref{skew form relation}) it costs a negative sign to uncross the legs of a jellyfish:
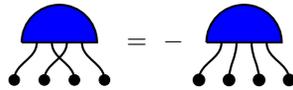
\[
    \begin{tikzpicture}[baseline = 12pt]
        %Draw vertices
        \draw \foreach \m in {-0.1,0.3,0.7,1.1} {
            (\m,0) node[circle, draw, fill=black]{}};
        %Draw legs
        \draw[-,thick] (-0.1, 0) to[out=up,in=down] (0.2, 0.5);
        \draw[-,thick] (0.3, 0) to[out=up,in=down] (0.6, 0.5);
        \draw[-,thick] (0.7, 0) to[out=up,in=down] (0.4, 0.5);
        \draw[-,thick] (1.1, 0) to[out=up,in=down] (0.8, 0.5);
        %Draw jellyfish
        \filldraw[fill=\jellycolor,-,thick] (0,0.5) to (1,0.5) to[out=up, in=right] (0.5,1) to[out=left, in=up] (0,0.5);
    \end{tikzpicture}
    ~=~-~
    \begin{tikzpicture}[baseline = 12pt]
        %Draw vertices
        \draw[-,thick] \foreach \m in {-0.1,0.3,0.7,1.1} {
            (\m,0) node[circle, draw, fill=black]{}};
        %Draw legs
        \draw[-,thick] (-0.1, 0) to[out=up,in=down] (0.2, 0.5);
        \draw[-,thick] (0.3, 0) to[out=up,in=down] (0.4, 0.5);
        \draw[-,thick] (0.7, 0) to[out=up,in=down] (0.6, 0.5);
        \draw[-,thick] (1.1, 0) to[out=up,in=down] (0.8, 0.5);
        %Draw jellyfish
        \filldraw[fill=\jellycolor,-,thick] (0,0.5) to (1,0.5) to[out=up, in=right] (0.5,1) to[out=left, in=up] (0,0.5);
    \end{tikzpicture}
\]
Note that crossing two dangling legs does not change the morphism in $\JP(n)$. As a consequence (since we assume $\operatorname{char}\k\not=2$) we get that any jellyfish diagram with more than one dangling leg is equal to zero. For example, 
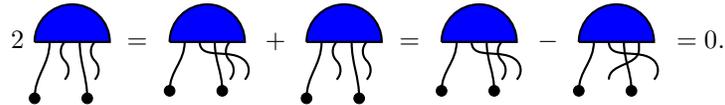
\[
    2\,\begin{tikzpicture}[baseline = 12pt]
        %Draw vertices
        \draw \foreach \m in {0,0.7} {
            (\m,-0.25) node[circle, draw, fill=black]{}};
        %Draw legs
        \draw[-,thick] (0, -0.25) to[out=up,in=down] (0.2, 0.5);
        \draw[-,thick] (0.7, -0.25) to[out=up,in=down] (0.6, 0.5);
        \draw[-,thick] (0.8, 0.5) to[out=-135,in=135] (0.8,0.25) to[out=-45,in=45] (0.8,0);
        \draw[-,thick] (0.4, 0.5) to[out=-135,in=135] (0.4,0.25) to[out=-45,in=45] (0.4,0);
        %Draw jellyfish
        \filldraw[fill=\jellycolor,-,thick] (0,0.5) to (1,0.5) to[out=up, in=right] (0.5,1) to[out=left, in=up] (0,0.5);
    \end{tikzpicture}
    ~=~\begin{tikzpicture}[baseline = 12pt]
        %Draw vertices
        \draw \foreach \m in {0,0.7} {
            (\m,-0.15) node[circle, draw, fill=black]{}};
        %Draw legs
        \draw[-,thick] (0, -0.15) to[out=up,in=down] (0.2, 0.5);
        \draw[-,thick] (0.7, -0.15) to[out=up,in=down] (0.6, 0.5);
        \draw[-,thick] (0.8, 0.5) to[out=-135,in=135] (0.8,0.25) to[out=-45,in=45] (0.8,0);
        \draw[-,thick] (0.4, 0.5) to[out=down,in=135] (1,0.25) to[out=-45,in=45] (1,0);
        %Draw jellyfish
        \filldraw[fill=\jellycolor,-,thick] (0,0.5) to (1,0.5) to[out=up, in=right] (0.5,1) to[out=left, in=up] (0,0.5);
    \end{tikzpicture}
    ~+~\begin{tikzpicture}[baseline = 12pt]
        %Draw vertices
        \draw \foreach \m in {0,0.7} {
            (\m,-0.25) node[circle, draw, fill=black]{}};
        %Draw legs
        \draw[-,thick] (0, -0.25) to[out=up,in=down] (0.2, 0.5);
        \draw[-,thick] (0.7, -0.25) to[out=up,in=down] (0.6, 0.5);
        \draw[-,thick] (0.8, 0.5) to[out=-135,in=135] (0.8,0.25) to[out=-45,in=45] (0.8,0);
        \draw[-,thick] (0.4, 0.5) to[out=-135,in=135] (0.4,0.25) to[out=-45,in=45] (0.4,0);
        %Draw jellyfish
        \filldraw[fill=\jellycolor,-,thick] (0,0.5) to (1,0.5) to[out=up, in=right] (0.5,1) to[out=left, in=up] (0,0.5);
    \end{tikzpicture}
    ~=~\begin{tikzpicture}[baseline = 12pt]
        %Draw vertices
        \draw \foreach \m in {0,0.7} {
            (\m,-0.15) node[circle, draw, fill=black]{}};
        %Draw legs
        \draw[-,thick] (0, -0.15) to[out=up,in=down] (0.2, 0.5);
        \draw[-,thick] (0.7, -0.15) to[out=up,in=down] (0.6, 0.5);
        \draw[-,thick] (0.8, 0.5) to[out=-135,in=135] (0.8,0.25) to[out=-45,in=45] (0.8,0);
        \draw[-,thick] (0.4, 0.5) to[out=down,in=135] (1,0.25) to[out=-45,in=45] (1,0);
        %Draw jellyfish
        \filldraw[fill=\jellycolor,-,thick] (0,0.5) to (1,0.5) to[out=up, in=right] (0.5,1) to[out=left, in=up] (0,0.5);
    \end{tikzpicture}
    ~-~\begin{tikzpicture}[baseline = 12pt]
        %Draw vertices
        \draw \foreach \m in {0,0.7} {
            (\m,-0.15) node[circle, draw, fill=black]{}};
        %Draw legs
        \draw[-,thick] (0, -0.15) to[out=up,in=down] (0.2, 0.5);
        \draw[-,thick] (0.7, -0.15) to[out=up,in=down] (0.6, 0.5);
        \draw[-,thick] (0.8, 0.5) to[out=down,in=up] (0.8,0.4) to[out=down,in=up] (0.4,0);
        \draw[-,thick] (0.4, 0.5) to[out=down,in=135] (1,0.25) to[out=-45,in=45] (1,0);
        %Draw jellyfish
        \filldraw[fill=\jellycolor,-,thick] (0,0.5) to (1,0.5) to[out=up, in=right] (0.5,1) to[out=left, in=up] (0,0.5);
    \end{tikzpicture}
    ~=0.
\]
In the computation above (\ref{skew form relation}) was only used to permute the legs of the right jellyfish. 
Similarly, any jellyfish diagram where two legs of the same jellyfish are connected is also zero. For example,
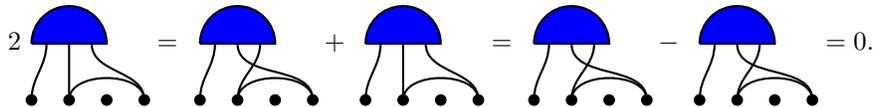
\[
    2\,\begin{tikzpicture}[baseline = 12pt]
        %Draw vertices
        \draw \foreach \m in {0,0.5,1,1.5} {
            (\m,-0.25) node[circle, draw, fill=black]{}};
        %Draw legs
        \draw[-,thick] (0,-0.25) to[out=up,in=down] (0.2, 0.5);
        \draw[-,thick] ( 0.5,-0.25) to[out=up,in=down] (0.5, 0.5);
        \draw[-,thick] ( 1.5,-0.25) to[out=up,in=down] (0.8, 0.5);
        \draw[-,thick] (0.5,-0.25) to[out=up,in=up] (1.5,-0.25);
        %Draw jellyfish
        \filldraw[fill=\jellycolor,-,thick] (0,0.5) to (1,0.5) to[out=up, in=right] (0.5,1) to[out=left, in=up] (0,0.5);
    \end{tikzpicture}
    =~\begin{tikzpicture}[baseline = 12pt]
        %Draw vertices
        \draw \foreach \m in {0,0.5,1,1.5} {
            (\m,-0.25) node[circle, draw, fill=black]{}};
        %Draw legs
        \draw[-,thick] (0,-0.25) to[out=up,in=down] (0.2, 0.5);
        \draw[-,thick] ( 0.5,-0.25) to[out=up,in=down] (0.8, 0.5);
        \draw[-,thick] ( 1.5,-0.25) to[out=up,in=down] (0.5, 0.5);
        \draw[-,thick] (0.5,-0.25) to[out=up,in=up] (1.5,-0.25);
        %Draw jellyfish
        \filldraw[fill=\jellycolor,-,thick] (0,0.5) to (1,0.5) to[out=up, in=right] (0.5,1) to[out=left, in=up] (0,0.5);
    \end{tikzpicture}
    +~\begin{tikzpicture}[baseline = 12pt]
        %Draw vertices
        \draw \foreach \m in {0,0.5,1,1.5} {
            (\m,-0.25) node[circle, draw, fill=black]{}};
        %Draw legs
        \draw[-,thick] (0,-0.25) to[out=up,in=down] (0.2, 0.5);
        \draw[-,thick] ( 0.5,-0.25) to[out=up,in=down] (0.5, 0.5);
        \draw[-,thick] ( 1.5,-0.25) to[out=up,in=down] (0.8, 0.5);
        \draw[-,thick] (0.5,-0.25) to[out=up,in=up] (1.5,-0.25);
        %Draw jellyfish
        \filldraw[fill=\jellycolor,-,thick] (0,0.5) to (1,0.5) to[out=up, in=right] (0.5,1) to[out=left, in=up] (0,0.5);
    \end{tikzpicture}
    =~\begin{tikzpicture}[baseline = 12pt]
        %Draw vertices
        \draw \foreach \m in {0,0.5,1,1.5} {
            (\m,-0.25) node[circle, draw, fill=black]{}};
        %Draw legs
        \draw[-,thick] (0,-0.25) to[out=up,in=down] (0.2, 0.5);
        \draw[-,thick] ( 0.5,-0.25) to[out=up,in=down] (0.8, 0.5);
        \draw[-,thick] ( 1.5,-0.25) to[out=up,in=down] (0.5, 0.5);
        \draw[-,thick] (0.5,-0.25) to[out=up,in=up] (1.5,-0.25);
        %Draw jellyfish
        \filldraw[fill=\jellycolor,-,thick] (0,0.5) to (1,0.5) to[out=up, in=right] (0.5,1) to[out=left, in=up] (0,0.5);
    \end{tikzpicture}
    -~\begin{tikzpicture}[baseline = 12pt]
        %Draw vertices
        \draw \foreach \m in {0,0.5,1,1.5} {
            (\m,-0.25) node[circle, draw, fill=black]{}};
        %Draw legs
        \draw[-,thick] (0,-0.25) to[out=up,in=down] (0.2, 0.5);
        \draw[-,thick] ( 0.5,-0.25) to[out=up,in=down] (0.8, 0.5);
        \draw[-,thick] ( 1.5,-0.25) to[out=up,in=down] (0.5, 0.5);
        \draw[-,thick] (0.5,-0.25) to[out=up,in=up] (1.5,-0.25);
        %Draw jellyfish
        \filldraw[fill=\jellycolor,-,thick] (0,0.5) to (1,0.5) to[out=up, in=right] (0.5,1) to[out=left, in=up] (0,0.5);
    \end{tikzpicture}
    =0.
\]

The \emph{jellyfish relation} (\ref{jellyfish relation}) is a bit more complicated. Its diagrammatic version is the following:
\[
    \begin{tikzpicture}[baseline = 12pt]
        %Draw bottom vertices
        \draw[-,thick]  \foreach \m in {1,2,2.5, 3.5} {
            (\m,0) node[circle, draw, fill=black]{}};
        %Draw edges
        \draw[-,thick] (1,0) to[out=up,in=down] (1.2,0.5);
        \draw[-,thick] (2,0) to[out=up,in=down] (1.8,0.5);
        \draw[-,thick] (2.5,0) to[out=up,in=down] (2.7,0.5);
        \draw[-,thick] (3.5,0) to[out=up,in=down] (3.3,0.5);
        %Draw jellyfish
        \filldraw[fill=\jellycolor,-,thick] (1,0.5) to (2,0.5) to[out=up, in=right] (1.5,1) to[out=left, in=up] (1,0.5);
        \filldraw[fill=\jellycolor,-,thick] (2.5,0.5) to (3.5,0.5) to[out=up, in=right] (3,1) to[out=left, in=up] (2.5,0.5);
        \draw (1.5,.25) node{$\cdots$};
        \draw (3,.25) node{$\cdots$};
    \end{tikzpicture}
    ~=\sum_{\sigma\in S_n}(-1)^\sigma~
    \begin{tikzpicture}[baseline = 12pt]
        %Draw bottom vertices
        \draw[-,thick]  \foreach \m in {1,2,2.5, 3.5} {
            (\m,0) node[circle, draw, fill=black]{}
            (\m, 0) to (\m, 0.25)
            };
        %Draw edges
        \draw[-,thick] (0.9,0.25) to (2.1,0.25) to (2.1,0.75) to (0.9,0.75) to (0.9,0.25);      
        \draw[-,thick] (1,0.75) to[out=up, in=left] (1.75,1.25) to[out=right, in=up] (2.5,0.75) to (2.5,0.25);
        \draw[-,thick] (2,0.75) to[out=up, in=left] (2.75,1.25) to[out=right, in=up] (3.5,0.75) to (3.5,0.25);
        \draw (1.5,.1) node{$\cdots$};
        \draw (1.5,.5) node{$\sigma$};
        \draw (3,.5) node{$\cdots$};  
    \end{tikzpicture}
\]
For example, when $n=2$ we have 
\[
    \begin{tikzpicture}[baseline = 8pt]
        %Draw bottom vertices
        \draw[-,thick]  \foreach \m in {1,...,4} {
            (0.5*\m,0) node[circle, draw, fill=black]{}};
        %Draw edges
        \draw[-,thick] (0.5,0) to[out=up,in=down] (0.6,0.4);
        \draw[-,thick] (1,0) to[out=up,in=down] (0.9,0.4);
        \draw[-,thick] (1.5,0) to[out=up,in=down] (1.6,0.4);
        \draw[-,thick] (2,0) to[out=up,in=down] (1.9,0.4);
        %Draw jellyfish
        \filldraw[fill=\jellycolor,-,thick] (0.4,0.4) to (1.1,0.4) to[out=up, in=right] (0.75,0.75) to[out=left, in=up] (0.4,0.4);
        \filldraw[fill=\jellycolor,-,thick] (1.4,0.4) to (2.1,0.4) to[out=up, in=right] (1.75,0.75) to[out=left, in=up] (1.4,0.4);
    \end{tikzpicture}
    ~=~\begin{tikzpicture}[baseline = 8pt]
        %Draw bottom vertices
        \draw \foreach \m in {1,...,4} {
            (0.5*\m,0) node[circle, draw, fill=black]{}
            };
        %Draw edges
        \draw[-,thick] (0.5,0) to[out=up, in=left] (1,0.8) to[out=right, in=up] (1.5,0);
        \draw[-,thick] (1,0) to[out=up, in=left] (1.5,0.8) to[out=right, in=up] (2,0);
    \end{tikzpicture}
    ~-~\begin{tikzpicture}[baseline = 8pt]
        %Draw bottom vertices
        \draw \foreach \m in {1,...,4} {
            (0.5*\m,0) node[circle, draw, fill=black]{}
            };
        %Draw jellyfish
        \draw[-,thick] (0.5,0) to[out=up, in=left] (1.25,0.8) to[out=right, in=up] (2,0);
        \draw[-,thick] (1,0) to[out=up, in=left] (1.25,0.4) to[out=right, in=up] (1.5,0);
    \end{tikzpicture}
\] 
The jellyfish relation explains how to reduce the number of jellyfish in any jellyfish diagram with more than one jellyfish. Moreover, the product (horizontal or vertical) of any two jellyfish diagrams each having exactly one jellyfish can be written in terms of partition diagrams. For example, if $n=2$ and we let  
\[D_1=~
    \begin{tikzpicture}[baseline = 12pt]
        %Draw bottom vertices
        \draw[-,thick]  \foreach \m in {1,...,2} {
            (\m,0) node[circle, draw, fill=black]{}
            (\m,1) node[circle, draw, fill=black]{}};
        %Draw edges
        \draw[-,thick] (2,0) to[out=up,in=down] (1.9,0.4);
        \draw[-,thick] (1.5,0.4) to[out=down,in=down] (1,0.4) to[out=up,in=down] (1,1);
        %Draw jellyfish
        \filldraw[fill=\jellycolor,-,thick] (1.3,0.4) to (2.1,0.4) to[out=up, in=right] (1.7,0.8) to[out=left, in=up] (1.3,0.4);
    \end{tikzpicture}
    \quad\text{and}\quad
    D_2=~
    \begin{tikzpicture}[baseline = 12pt]
        %Draw bottom vertices
        \draw[-,thick]  \foreach \m in {1,...,2} {
            (\m,0) node[circle, draw, fill=black]{}
            (\m,1) node[circle, draw, fill=black]{}};
        %Draw edges
        \draw[-,thick] (1,0) to[out=up,in=down] (1,1);
        \draw[-,thick] (2,0) to[out=up,in=down] (1.5,0.4);
        \draw[-,thick] (1.9,0.4) to[out=down,in=down] (2.3,0.4) to[out=up,in=down] (2,1);
        %Draw jellyfish
        \filldraw[fill=\jellycolor,-,thick] (1.3,0.4) to (2.1,0.4) to[out=up, in=right] (1.7,0.8) to[out=left, in=up] (1.3,0.4);
    \end{tikzpicture}   
    \quad\text{then}
\]
\[D_2\circ D_1=~
    \begin{tikzpicture}[baseline = 21pt, scale=0.8]
        %Draw bottom vertices
        \draw[-,thick]  \foreach \m in {1,...,2} {
            (\m,0) node[circle, draw, fill=black]{}
            (\m,1) node[circle, draw, fill=black]{}
            (\m,2) node[circle, draw, fill=black]{}};
        %Draw edges
        \draw[-,thick] (1,1) to[out=up,in=down] (1,2);
        \draw[-,thick] (2,1) to[out=up,in=down] (1.5,1.4);
        \draw[-,thick] (1.9,1.4) to[out=down,in=down] (2.3,1.4) to[out=up,in=down] (2,2);
        %Draw jellyfish
        \filldraw[fill=\jellycolor,-,thick] (1.3,1.4) to (2.1,1.4) to[out=up, in=right] (1.7,1.8) to[out=left, in=up] (1.3,1.4);
        %Draw edges
        \draw[-,thick] (2,0) to[out=up,in=down] (1.9,0.4);
        \draw[-,thick] (1.5,0.4) to[out=down,in=down] (1,0.4) to[out=up,in=down] (1,1);
        %Draw jellyfish
        \filldraw[fill=\jellycolor,-,thick] (1.3,0.4) to (2.1,0.4) to[out=up, in=right] (1.7,0.8) to[out=left, in=up] (1.3,0.4);
    \end{tikzpicture}
    ~=~
    \begin{tikzpicture}[baseline = 12pt, scale=0.8]
        %Draw bottom vertices
        \draw[-,thick]  \foreach \m in {1,...,2} {
            (\m,0) node[circle, draw, fill=black]{}
            (\m,1.5) node[circle, draw, fill=black]{}};
        %Draw edges
        \draw[-,thick] (2,0) to[out=150,in=down] (1.2,0.5);
        \draw[-,thick] (0.8,0.5) to[out=down,in=down] (0.4,0.5) to[out=up,in=down] (1,1.5);
        \draw[-,thick] (2.2,0.5) to[out=down,in=down] (2.6,0.5) to[out=up,in=down] (2,1.5);
        \draw[-,thick] (1.8, 0.5) to[out=-135,in=135] (1.8,0.375) to[out=-45,in=45] (1.8,0.25);
        %Draw jellyfish
        \filldraw[fill=\jellycolor,-,thick] (1.6,0.5) to (2.4,0.5) to[out=up, in=right] (2,0.9) to[out=left, in=up] (1.6,0.5);
        \filldraw[fill=\jellycolor,-,thick] (0.6,0.5) to (1.4,0.5) to[out=up, in=right] (1,0.9) to[out=left, in=up] (0.6,0.5);
    \end{tikzpicture}
    ~=~
    \begin{tikzpicture}[baseline = 12pt, scale=0.8]
        %Draw bottom vertices
        \draw[-,thick]  \foreach \m in {1,...,2} {
            (\m,0) node[circle, draw, fill=black]{}
            (\m,1.5) node[circle, draw, fill=black]{}};
        %Draw edges
        \draw[-,thick] (2,0) to[out=150,in=down] (1.2,0.5);
        \draw[-,thick] (0.8,0.5) to[out=down,in=down] (0.4,0.5) to[out=up,in=down] (1,1.5);
        \draw[-,thick] (2.2,0.5) to[out=down,in=down] (2.6,0.5) to[out=up,in=down] (2,1.5);
        \draw[-,thick] (1.8, 0.5) to[out=down,in=135] (1.8,0.375) to[out=-45,in=45] (1.8,0.25);
        \draw[-,thick] (0.8,0.5) to[out=up,in=up] (1.8,0.5);
        \draw[-,thick] (1.2,0.5) to[out=up,in=up] (2.2,0.5);
    \end{tikzpicture}
    ~-~
    \begin{tikzpicture}[baseline = 12pt, scale=0.8]
        %Draw bottom vertices
        \draw[-,thick]  \foreach \m in {1,...,2} {
            (\m,0) node[circle, draw, fill=black]{}
            (\m,1.5) node[circle, draw, fill=black]{}};
        %Draw edges
        \draw[-,thick] (2,0) to[out=150,in=down] (1.2,0.5);
        \draw[-,thick] (0.8,0.5) to[out=down,in=down] (0.4,0.5) to[out=up,in=down] (1,1.5);
        \draw[-,thick] (2.2,0.5) to[out=down,in=down] (2.6,0.5) to[out=up,in=down] (2,1.5);
        \draw[-,thick] (1.8, 0.5) to[out=down,in=135] (1.8,0.375) to[out=-45,in=45] (1.8,0.25);
        \draw[-,thick] (0.8,0.5) to[out=up,in=up] (2.2,0.5);
        \draw[-,thick] (1.2,0.5) to[out=up,in=up] (1.8,0.5);
    \end{tikzpicture}
    ~=~
    \begin{tikzpicture}[baseline = 5pt, scale=0.5]
        %Draw bottom vertices
        \draw[-,thick]  \foreach \m in {1,...,2} {
            (\m,0) node[circle, draw, fill=black]{}
            (\m,1) node[circle, draw, fill=black]{}};
        %Draw edges
        \draw[-,thick] (2,0) to (2,1);
    \end{tikzpicture}
    ~~-~~
    \begin{tikzpicture}[baseline = 5pt, scale=0.5]
        %Draw bottom vertices
        \draw[-,thick]  \foreach \m in {1,...,2} {
            (\m,0) node[circle, draw, fill=black]{}
            (\m,1) node[circle, draw, fill=black]{}};
        %Draw edges
        \draw[-,thick] (1,1) to[out=down,in=left] (1.5,0.5) to[out=right,in=down] (2,1);
    \end{tikzpicture}
\]
Here is an example in $\JP(3)$:
\[\begin{array}{rl}
    \begin{tikzpicture}[baseline = 12pt, scale=0.75]
        %Draw bottom vertices
        \draw[-,thick] \foreach \m in {0.4,1,1.6} {
            (\m,-0.2) node[circle, draw, fill=black]{}
            (\m,1.7) node[circle, draw, fill=black]{}
            };
        %Draw legs
        \draw[-,thick] (0.4, -0.2) to[out=up,in=down] (0.5, 0.15);
        \draw[-,thick] (1, -0.2) to[out=up,in=down] (1, 0.15);
        \draw[-,thick] (1.6, -0.2) to[out=up,in=down] (1.5, 0.15);
        \draw[-,thick] (0.4, 1.7) to[out=down,in=up] (0.5, 1.35);
        \draw[-,thick] (1, 1.7) to[out=down,in=up] (1, 1.35);
        \draw[-,thick] (1.6, 1.7) to[out=down,in=up] (1.5, 1.35);
        %Draw jellyfish
        \filldraw[fill=\jellycolor,-,thick] (0.3,0.15) to (1.7,0.15) to[out=up, in=right] (1,0.6) to[out=left, in=up] (0.3,0.15);
        \filldraw[fill=\jellycolor,-,thick] (0.3,1.35) to (1.7,1.35) to[out=down, in=right] (1,0.9) to[out=left, in=down] (0.3,1.35);
    \end{tikzpicture}
    \!
    &=~
    \begin{tikzpicture}[baseline = 12pt, scale=0.75]
        %Draw bottom vertices
        \draw[-,thick] \foreach \m in {1,...,3} {
            (0.5*\m,-0.2) node[circle, draw, fill=black]{}
            (0.5*\m, -0.2) to[out=up, in=down] (0.5*\m-1, 0.75)
            (0.5*\m+1.5,1.7) node[circle, draw, fill=black]{}
            (0.5*\m+1.5, 1.7) to[out=down, in=up] (0.5*\m+2.4, 0.75) to[out=down, in=down] (2.7-0.5*\m, 0.75)
            };
            %Draw jellyfish
            \filldraw[fill=\jellycolor,-,thick] (-0.7,0.75) to (0.7,0.75) to[out=up, in=right] (0,1.2) to[out=left, in=up] (-0.7,0.75);
            \filldraw[fill=\jellycolor,-,thick] (1,0.75) to (2.4,0.75) to[out=up, in=right] (1.7,1.2) to[out=left, in=up] (1,0.75);
    \end{tikzpicture}
    \\
    &=\sum\limits_{\sigma\in S_3}(-1)^\sigma~
    \begin{tikzpicture}[baseline = 12pt, scale=0.75]
        %Draw bottom vertices
        \draw[-,thick] \foreach \m in {1,...,3} {
            (0.5*\m,-0.2) node[circle, draw, fill=black]{}
            (0.5*\m, -0.2) to[out=up, in=down] (0.5*\m-1, 0.5)
            (0.5*\m+1.5,1.7) node[circle, draw, fill=black]{}
            (0.5*\m+1.5, 1.7) to[out=down, in=up] (0.5*\m+2.4, 0.75) to[out=down, in=down] (2.7-0.5*\m, 0.75) to[out=up, in=up] (1-0.5*\m, 1)
            };
            \draw[-,thick] (-0.7,0.5) to (0.7,0.5) to (0.7,1) to (-0.7,1) to (-0.7,0.5);
            \draw (0,0.75) node{$\sigma$};
    \end{tikzpicture}
    \\
    \\
    &=
    \begin{tikzpicture}[baseline = 8pt, scale=0.75]
        %Draw bottom vertices
        \draw[-,thick] \foreach \m in {1,...,3} {
            (0.5*\m,0) node[circle, draw, fill=black]{}
            (0.5*\m,1) node[circle, draw, fill=black]{}
            };
            \draw[-,thick] (0.5,0) to[out=up,in=down] (1.5,1);
            \draw[-,thick] (0.5,1) to[out=down,in=up] (1.5,0);
        \draw[-,thick] (1,0) to[out=up, in=down] (0.75, 0.5) to[out=up, in=down] (1,1);
    \end{tikzpicture}
    ~+~
    \begin{tikzpicture}[baseline = 8pt, scale=0.75]
        %Draw bottom vertices
        \draw[-,thick] \foreach \m in {1,...,3} {
            (0.5*\m,0) node[circle, draw, fill=black]{}
            (0.5*\m,1) node[circle, draw, fill=black]{}
            };
            \draw[-,thick] (0.5,0) to[out=up,in=down] (1,1);
            \draw[-,thick] (0.5,1) to[out=down,in=up] (1,0);
        \draw[-,thick] (1.5,0) to[out=up,in=down] (1.5,1);
    \end{tikzpicture}
    ~+~
    \begin{tikzpicture}[baseline = 8pt, scale=0.75]
        %Draw bottom vertices
        \draw[-,thick] \foreach \m in {1,...,3} {
            (0.5*\m,0) node[circle, draw, fill=black]{}
            (0.5*\m,1) node[circle, draw, fill=black]{}
            };
            \draw[-,thick] (1.5,0) to[out=up,in=down] (1,1);
            \draw[-,thick] (1.5,1) to[out=down,in=up] (1,0);
        \draw[-,thick] (0.5,0) to[out=up,in=down] (0.5,1);
    \end{tikzpicture}
    ~-~
    \begin{tikzpicture}[baseline = 8pt, scale=0.75]
        %Draw bottom vertices
        \draw[-,thick] \foreach \m in {1,...,3} {
            (0.5*\m,0) node[circle, draw, fill=black]{}
            (0.5*\m,1) node[circle, draw, fill=black]{}
            };
            \draw[-,thick] (1.5,0) to[out=up,in=down] (0.5,1);
            \draw[-,thick] (1,0) to[out=up,in=down] (1.5,1);
        \draw[-,thick] (0.5,0) to[out=up,in=down] (1,1);
    \end{tikzpicture}
    ~-~
    \begin{tikzpicture}[baseline = 8pt, scale=0.75]
        %Draw bottom vertices
        \draw[-,thick] \foreach \m in {1,...,3} {
            (0.5*\m,0) node[circle, draw, fill=black]{}
            (0.5*\m,1) node[circle, draw, fill=black]{}
            };
            \draw[-,thick] (1.5,1) to[out=down,in=up] (0.5,0);
            \draw[-,thick] (1,1) to[out=down,in=up] (1.5,0);
        \draw[-,thick] (0.5,1) to[out=down,in=up] (1,0);
    \end{tikzpicture}
    ~-~
    \begin{tikzpicture}[baseline = 8pt, scale=0.75]
        %Draw bottom vertices
        \draw[-,thick] \foreach \m in {1,...,3} {
            (0.5*\m,0) node[circle, draw, fill=black]{}
            (0.5*\m,1) node[circle, draw, fill=black]{}
            (0.5*\m,0) to[out=up,in=down] (0.5*\m,1)
            };
    \end{tikzpicture}
\end{array}\]

\begin{remark}
A computation similar to the one above shows 
\[\jelly^*\circ \jelly=(-1)^{\lfloor n/2\rfloor}\sum_{\sigma\in S_n}(-1)^\sigma\sigma,\] 
which is a scalar multiple of the primitive idempotent corresponding to the sign representation. The equation above could be used as an alternative to the jellyfish relation (\ref{jellyfish relation}) in the definition of $\JP(n)$. 
\end{remark}

We conclude this section with two examples that show partition diagrams are not linearly independent when viewed as morphisms in $\JP(n)$. The first example will appear again in the proof of Lemma \ref{reduce Y} in more generality. 

\begin{example}\label{example p dependent}
Let $D:3\to0$ denote the following jellyfish diagram in $\JP(2)$:
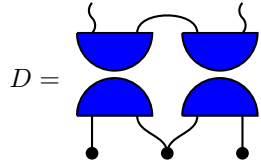
\[
    D=~
    \begin{tikzpicture}[baseline = 25pt]
        %Draw bottom vertices
        \draw[-,thick]  \foreach \m in {1,...,3} {
            (\m,0) node[circle, draw, fill=black]{}};
        %Draw edges
        \draw[-,thick] (1,0) to[out=up,in=down] (1,0.5);
        \draw[-,thick] (2,0) to[out=up,in=down] (1.6,0.5);
        \draw[-,thick] (2,0) to[out=up,in=down] (2.4,0.5);
        \draw[-,thick] (3,0) to[out=up,in=down] (3,0.5);
        \draw[-,thick] (1.6,1.6) to[out=up,in=up] (2.4,1.6);
        \draw[-,thick] (3, 2) to[out=-135,in=135] (3,1.8) to[out=-45,in=45] (3,1.6);
        \draw[-,thick] (1, 2) to[out=-135,in=135] (1,1.8) to[out=-45,in=45] (1,1.6);
        %Draw jellyfish
        \filldraw[fill=\jellycolor,-,thick] (0.8,0.5) to (1.8,0.5) to[out=up, in=right] (1.3,1) to[out=left, in=up] (0.8,0.5);
        \filldraw[fill=\jellycolor,-,thick] (2.2,0.5) to (3.2,0.5) to[out=up, in=right] (2.7,1) to[out=left, in=up] (2.2,0.5);
        \filldraw[fill=\jellycolor,-,thick] (0.8,1.6) to (1.8,1.6) to[out=down, in=right] (1.3,1.1) to[out=left, in=down] (0.8,1.6);
        \filldraw[fill=\jellycolor,-,thick] (2.2,1.6) to (3.2,1.6) to[out=down, in=right] (2.7,1.1) to[out=left, in=down] (2.2,1.6);
    \end{tikzpicture}
\]
Applying the jellyfish relation (\ref{jellyfish relation}) to the top and bottom pairs of jellyfish respectively gives 
\[\begin{array}{rcccccccc}
    D= && 
    \begin{tikzpicture}[baseline = 20pt, scale=0.8]
        %Draw bottom vertices
        \draw[-,thick]  \foreach \m in {1,...,3} {
            (\m,0) node[circle, draw, fill=black]{}};
        %Draw edges
        \draw[-,thick] (1,0) to[out=up,in=down] (1,0.3);
        \draw[-,thick] (2,0) to[out=up,in=down] (1.6,0.4);
        \draw[-,thick] (2,0) to[out=up,in=down] (2.4,0.4);
        \draw[-,thick] (3,0) to[out=up,in=down] (3,0.3);
        \draw[-,thick] (1.6,1.6) to[out=up,in=up] (2.4,1.6);
        \draw[-,thick] (3, 2) to[out=-135,in=135] (3,1.8) to[out=-45,in=up] (3,1.6);
        \draw[-,thick] (1, 2) to[out=-135,in=135] (1,1.8) to[out=-45,in=up] (1,1.6);
        
        \draw[-,thick] (1,1.6) to[out=down,in=down] (2.4,1.6);
        \draw[-,thick] (1.6,1.6) to[out=down,in=down] (3,1.6);
        \draw[-,thick] (1,0.3) to[out=up,in=up] (2.4,0.4);
        \draw[-,thick] (1.6,0.4) to[out=up,in=up] (3,0.3);
    \end{tikzpicture}
    &-&
    \begin{tikzpicture}[baseline = 20pt, scale=0.8]
        %Draw bottom vertices
        \draw[-,thick]  \foreach \m in {1,...,3} {
            (\m,0) node[circle, draw, fill=black]{}};
        %Draw edges
        \draw[-,thick] (1,0) to[out=up,in=down] (1,0.3);
        \draw[-,thick] (2,0) to[out=up,in=down] (1.6,0.4);
        \draw[-,thick] (2,0) to[out=up,in=down] (2.4,0.4);
        \draw[-,thick] (3,0) to[out=up,in=down] (3,0.3);
        \draw[-,thick] (1.6,1.6) to[out=up,in=up] (2.4,1.6);
        \draw[-,thick] (3, 2) to[out=-135,in=135] (3,1.8) to[out=-45,in=up] (3,1.6);
        \draw[-,thick] (1, 2) to[out=-135,in=135] (1,1.8) to[out=-45,in=up] (1,1.6);
        
        \draw[-,thick] (1,1.6) to[out=down,in=down] (2.4,1.6);
        \draw[-,thick] (1.6,1.6) to[out=down,in=down] (3,1.6);
        \draw[-,thick] (1,0.3) to[out=up,in=up] (3,0.3);
        \draw[-,thick] (1.6,0.4) to[out=up,in=up] (2.4,0.4);
    \end{tikzpicture}
    &-&
    \begin{tikzpicture}[baseline = 20pt, scale=0.8]
        %Draw bottom vertices
        \draw[-,thick]  \foreach \m in {1,...,3} {
            (\m,0) node[circle, draw, fill=black]{}};
        %Draw edges
        \draw[-,thick] (1,0) to[out=up,in=down] (1,0.3);
        \draw[-,thick] (2,0) to[out=up,in=down] (1.6,0.4);
        \draw[-,thick] (2,0) to[out=up,in=down] (2.4,0.4);
        \draw[-,thick] (3,0) to[out=up,in=down] (3,0.3);
        \draw[-,thick] (1.6,1.6) to[out=up,in=up] (2.4,1.6);
        \draw[-,thick] (3, 2) to[out=-135,in=135] (3,1.8) to[out=-45,in=up] (3,1.6);
        \draw[-,thick] (1, 2) to[out=-135,in=135] (1,1.8) to[out=-45,in=up] (1,1.6);
        
        \draw[-,thick] (1,1.6) to[out=down,in=down] (3,1.6);
        \draw[-,thick] (1.6,1.6) to[out=down,in=down] (2.4,1.6);
        \draw[-,thick] (1,0.3) to[out=up,in=up] (2.4,0.4);
        \draw[-,thick] (1.6,0.4) to[out=up,in=up] (3,0.3);
    \end{tikzpicture}
    &+&
    \begin{tikzpicture}[baseline = 20pt, scale=0.8]
        %Draw bottom vertices
        \draw[-,thick]  \foreach \m in {1,...,3} {
            (\m,0) node[circle, draw, fill=black]{}};
        %Draw edges
        \draw[-,thick] (1,0) to[out=up,in=down] (1,0.3);
        \draw[-,thick] (2,0) to[out=up,in=down] (1.6,0.4);
        \draw[-,thick] (2,0) to[out=up,in=down] (2.4,0.4);
        \draw[-,thick] (3,0) to[out=up,in=down] (3,0.3);
        \draw[-,thick] (1.6,1.6) to[out=up,in=up] (2.4,1.6);
        \draw[-,thick] (3, 2) to[out=-135,in=135] (3,1.8) to[out=-45,in=up] (3,1.6);
        \draw[-,thick] (1, 2) to[out=-135,in=135] (1,1.8) to[out=-45,in=up] (1,1.6);
        
        \draw[-,thick] (1,1.6) to[out=down,in=down] (3,1.6);
        \draw[-,thick] (1.6,1.6) to[out=down,in=down] (2.4,1.6);
        \draw[-,thick] (1,0.3) to[out=up,in=up] (3,0.3);
        \draw[-,thick] (1.6,0.4) to[out=up,in=up] (2.4,0.4);
    \end{tikzpicture}
    \\
    \\
    =&2&
    \begin{tikzpicture}[baseline = 0pt, scale=0.8]
        %Draw bottom vertices
        \draw[-,thick]  \foreach \m in {1,...,3} {
            (\m,0) node[circle, draw, fill=black]{}};
        %Draw edges
        \draw[-,thick] (1,0) to[out=up,in=left] (1.5,0.5) to[out=right,in=up] (2,0) to[out=up,in=left] (2.5,0.5) to[out=right,in=up] (3,0);
    \end{tikzpicture}
    &-2&
    \begin{tikzpicture}[baseline = 0pt, scale=0.8]
        %Draw bottom vertices
        \draw[-,thick]  \foreach \m in {1,...,3} {
            (\m,0) node[circle, draw, fill=black]{}};
        %Draw edges
        \draw[-,thick] (1,0) to[out=up,in=up] (3,0);
    \end{tikzpicture}
    &-4&
    \begin{tikzpicture}[baseline = 0pt, scale=0.8]
        %Draw bottom vertices
        \draw[-,thick]  \foreach \m in {1,...,3} {
            (\m,0) node[circle, draw, fill=black]{}};
        %Draw edges
        \draw[-,thick] (1,0) to[out=up,in=left] (1.5,0.5) to[out=right,in=up] (2,0) to[out=up,in=left] (2.5,0.5) to[out=right,in=up] (3,0);
    \end{tikzpicture}
    &+4&
    \begin{tikzpicture}[baseline = 0pt, scale=0.8]
        %Draw bottom vertices
        \draw[-,thick]  \foreach \m in {1,...,3} {
            (\m,0) node[circle, draw, fill=black]{}};
        %Draw edges
        \draw[-,thick] (1,0) to[out=up,in=up] (3,0);
    \end{tikzpicture}
    \\[3pt]
    =&2&
    \begin{tikzpicture}[baseline = 0pt, scale=0.8]
        %Draw bottom vertices
        \draw[-,thick]  \foreach \m in {1,...,3} {
            (\m,0) node[circle, draw, fill=black]{}};
        %Draw edges
        \draw[-,thick] (1,0) to[out=up,in=up] (3,0);
    \end{tikzpicture}
    &-2&
    \begin{tikzpicture}[baseline = 0pt, scale=0.8]
        %Draw bottom vertices
        \draw[-,thick]  \foreach \m in {1,...,3} {
            (\m,0) node[circle, draw, fill=black]{}};
        %Draw edges
        \draw[-,thick] (1,0) to[out=up,in=left] (1.5,0.5) to[out=right,in=up] (2,0) to[out=up,in=left] (2.5,0.5) to[out=right,in=up] (3,0);
    \end{tikzpicture}
\end{array}\]
On the other hand, applying the jellyfish relation (\ref{jellyfish relation}) to the left and right  pairs of jellyfish respectively gives  
\[\begin{array}{rccccccccccc}
    D=& 
    \begin{tikzpicture}[baseline = 20pt, scale=0.8]
        %Draw bottom vertices
        \draw[-,thick]  \foreach \m in {1,...,3} {
            (\m,0) node[circle, draw, fill=black]{}};
        %Draw edges
        \draw[-,thick] (1,0) to[out=up,in=down] (1,0.3);
        \draw[-,thick] (2,0) to[out=up,in=down] (1.6,0.5);
        \draw[-,thick] (2,0) to[out=up,in=down] (2.4,0.5);
        \draw[-,thick] (3,0) to[out=up,in=down] (3,0.3);
        \draw[-,thick] (1.6,1.6) to[out=up,in=up] (2.4,1.6);
        \draw[-,thick] (3, 2) to[out=-135,in=135] (3,1.8) to[out=-45,in=up] (3,1.6);
        \draw[-,thick] (1, 2) to[out=-135,in=135] (1,1.8) to[out=-45,in=up] (1,1.6);
        
        \draw[-,thick] (1,0.3) to[out=up,in=down] (1.6,1.6);
        \draw[-,thick] (1.6,0.5) to[out=up,in=down] (1,1.6);
        \draw[-,thick] (2.4,0.5) to[out=up,in=down] (3,1.6);
        \draw[-,thick] (3,0.3)to[out=up,in=down] (2.4,1.6);
    \end{tikzpicture}
    &-&
    \begin{tikzpicture}[baseline = 20pt, scale=0.8]
        %Draw bottom vertices
        \draw[-,thick]  \foreach \m in {1,...,3} {
            (\m,0) node[circle, draw, fill=black]{}};
        %Draw edges
        \draw[-,thick] (1,0) to[out=up,in=down] (1,0.3);
        \draw[-,thick] (2,0) to[out=up,in=down] (1.6,0.5);
        \draw[-,thick] (2,0) to[out=up,in=down] (2.4,0.5);
        \draw[-,thick] (3,0) to[out=up,in=down] (3,0.3);
        \draw[-,thick] (1.6,1.6) to[out=up,in=up] (2.4,1.6);
        \draw[-,thick] (3, 2) to[out=-135,in=135] (3,1.8) to[out=-45,in=up] (3,1.6);
        \draw[-,thick] (1, 2) to[out=-135,in=135] (1,1.8) to[out=-45,in=up] (1,1.6);
        
        \draw[-,thick] (1,0.3) to[out=up,in=down] (1.6,1.6);
        \draw[-,thick] (1.6,0.5) to[out=up,in=down] (1,1.6);
        \draw[-,thick] (2.4,0.5) to[out=up,in=down] (2.4,1.6);
        \draw[-,thick] (3,0.3)to[out=up,in=down] (3,1.6);
    \end{tikzpicture}
    &-&
    \begin{tikzpicture}[baseline = 20pt, scale=0.8]
        %Draw bottom vertices
        \draw[-,thick]  \foreach \m in {1,...,3} {
            (\m,0) node[circle, draw, fill=black]{}};
        %Draw edges
        \draw[-,thick] (1,0) to[out=up,in=down] (1,0.3);
        \draw[-,thick] (2,0) to[out=up,in=down] (1.6,0.5);
        \draw[-,thick] (2,0) to[out=up,in=down] (2.4,0.5);
        \draw[-,thick] (3,0) to[out=up,in=down] (3,0.3);
        \draw[-,thick] (1.6,1.6) to[out=up,in=up] (2.4,1.6);
        \draw[-,thick] (3, 2) to[out=-135,in=135] (3,1.8) to[out=-45,in=up] (3,1.6);
        \draw[-,thick] (1, 2) to[out=-135,in=135] (1,1.8) to[out=-45,in=up] (1,1.6);
        
        \draw[-,thick] (1,0.3) to[out=up,in=down] (1,1.6);
        \draw[-,thick] (1.6,0.5) to[out=up,in=down] (1.6,1.6);
        \draw[-,thick] (2.4,0.5) to[out=up,in=down] (3,1.6);
        \draw[-,thick] (3,0.3)to[out=up,in=down] (2.4,1.6);
    \end{tikzpicture}
    &+&
    \begin{tikzpicture}[baseline = 20pt, scale=0.8]
        %Draw bottom vertices
        \draw[-,thick]  \foreach \m in {1,...,3} {
            (\m,0) node[circle, draw, fill=black]{}};
        %Draw edges
        \draw[-,thick] (1,0) to[out=up,in=down] (1,0.3);
        \draw[-,thick] (2,0) to[out=up,in=down] (1.6,0.5);
        \draw[-,thick] (2,0) to[out=up,in=down] (2.4,0.5);
        \draw[-,thick] (3,0) to[out=up,in=down] (3,0.3);
        \draw[-,thick] (1.6,1.6) to[out=up,in=up] (2.4,1.6);
        \draw[-,thick] (3, 2) to[out=-135,in=135] (3,1.8) to[out=-45,in=up] (3,1.6);
        \draw[-,thick] (1, 2) to[out=-135,in=135] (1,1.8) to[out=-45,in=up] (1,1.6);
        
        \draw[-,thick] (1,0.3) to[out=up,in=down] (1,1.6);
        \draw[-,thick] (1.6,0.5) to[out=up,in=down] (1.6,1.6);
        \draw[-,thick] (2.4,0.5) to[out=up,in=down] (2.4,1.6);
        \draw[-,thick] (3,0.3)to[out=up,in=down] (3,1.6);
    \end{tikzpicture}
    \\
    \\
    =&
    \begin{tikzpicture}[baseline = 0pt, scale=0.8]
        %Draw bottom vertices
        \draw[-,thick]  \foreach \m in {1,...,3} {
            (\m,0) node[circle, draw, fill=black]{}};
        %Draw edges
        \draw[-,thick] (1,0) to[out=up,in=up] (3,0);
    \end{tikzpicture}
    &-&
    \begin{tikzpicture}[baseline = 0pt, scale=0.8]
        %Draw bottom vertices
        \draw[-,thick]  \foreach \m in {1,...,3} {
            (\m,0) node[circle, draw, fill=black]{}};
        %Draw edges
        \draw[-,thick] (1,0) to[out=up,in=left] (1.5,0.5) to[out=right,in=up] (2,0);
    \end{tikzpicture}
    &-&
    \begin{tikzpicture}[baseline = 0pt, scale=0.8]
        %Draw bottom vertices
        \draw[-,thick]  \foreach \m in {1,...,3} {
            (\m,0) node[circle, draw, fill=black]{}};
        %Draw edges
        \draw[-,thick] (2,0) to[out=up,in=left] (2.5,0.5) to[out=right,in=up] (3,0);
    \end{tikzpicture}
    &+&
    \begin{tikzpicture}[baseline = 0pt, scale=0.8]
        %Draw bottom vertices
        \draw[-,thick]  \foreach \m in {1,...,3} {
            (\m,0) node[circle, draw, fill=black]{}};
    \end{tikzpicture}
\end{array}\]
Subtracting the two expressions for $D$ computed above shows us that the following linear combination of partition diagrams is zero in $\JP(2)$:
\begin{equation}\label{P dependent example}
    \begin{tikzpicture}[baseline = 0pt, scale=0.5]
        %Draw bottom vertices
        \draw[-,thick]  \foreach \m in {1,...,3} {
            (\m,0) node[circle, draw, fill=black]{}};
    \end{tikzpicture}\quad-\quad
    \begin{tikzpicture}[baseline = 0pt, scale=0.5]
        %Draw bottom vertices
        \draw[-,thick]  \foreach \m in {1,...,3} {
            (\m,0) node[circle, draw, fill=black]{}};
        %Draw edges
        \draw[-,thick] (1,0) to[out=up,in=up] (3,0);
    \end{tikzpicture}\quad-\quad
    \begin{tikzpicture}[baseline = 0pt, scale=0.5]
        %Draw bottom vertices
        \draw[-,thick]  \foreach \m in {1,...,3} {
            (\m,0) node[circle, draw, fill=black]{}};
        %Draw edges
        \draw[-,thick] (1,0) to[out=up,in=left] (1.5,0.5) to[out=right,in=up] (2,0);
    \end{tikzpicture}\quad-\quad
    \begin{tikzpicture}[baseline = 0pt, scale=0.5]
        %Draw bottom vertices
        \draw[-,thick]  \foreach \m in {1,...,3} {
            (\m,0) node[circle, draw, fill=black]{}};
        %Draw edges
        \draw[-,thick] (2,0) to[out=up,in=left] (2.5,0.5) to[out=right,in=up] (3,0);
    \end{tikzpicture}\quad+2\quad
    \begin{tikzpicture}[baseline = 0pt, scale=0.5]
        %Draw bottom vertices
        \draw[-,thick]  \foreach \m in {1,...,3} {
            (\m,0) node[circle, draw, fill=black]{}};
        %Draw edges
        \draw[-,thick] (1,0) to[out=up,in=left] (1.5,0.5) to[out=right,in=up] (2,0) to[out=up,in=left] (2.5,0.5) to[out=right,in=up] (3,0);
    \end{tikzpicture}
\end{equation}
Notice that (\ref{P dependent example}) is equal to $x_Y$ where 
\[Y=~
    \begin{tikzpicture}[baseline = 0pt, scale=0.5]
        %Draw bottom vertices
        \draw[-,thick]  \foreach \m in {1,...,3} {
            (\m,0) node[circle, draw, fill=black]{}};
    \end{tikzpicture}
\]
By Theorem \ref{thm: Phi is full and kernel}, $\Phi(x_Y)=0$ since $Y$ has more than $n=2$ parts. Indeed, since $\Res^{S_n}_{A_n}$ is faithful and (\ref{square of functors}) commutes, any linear dependency among partition diagrams in $\JP(n)$ must come from the kernel of $\Phi$. It will follow from Theorem \ref{Psi is faithful} that if $n>1$ and either $\operatorname{char}\k=0$ or $\operatorname{char}\k\geq n$, then all linear dependencies among partition diagrams in $\JP(n)$ come from the kernel of $\Phi$.  
\end{example}

% The following example is related to the nonassociativity of even Brauer algebras mentioned in \S\ref{jellyfish Brauer}.

\begin{example}\label{Grood non-assoc}
Let $\alpha$ and $\beta$ denote the following jellyfish diagrams:
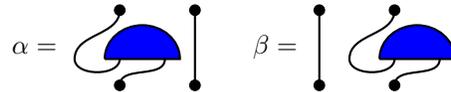
\[\alpha=~\begin{tikzpicture}[baseline = 12pt]
        %Draw bottom vertices
        \draw \foreach \m in {1,...,2} {
            (\m,0) node[circle, draw, fill=black,inner sep=0pt, minimum width=4pt]{}
            (\m,1) node[circle, draw, fill=black,inner sep=0pt, minimum width=4pt]{}
            };
        %Draw jellyfish
        \draw[-,thick] (2,0) to (2,1);
        \filldraw[fill=\jellycolor,-,thick] (0.8,0.35) to (1.8,0.35) to[out=up,in=right] (1.3,0.8) to[out=left,in=up] (0.8,0.35);
        \draw[-,thick] (1,0) to[out=up,in=down] (1.6,0.35);
        \draw[-,thick] (1,1) to[out=down,in=up] (0.4,0.35) to[out=down,in=down] (1,0.35);
    \end{tikzpicture}
    \qquad
    \beta=~
    \begin{tikzpicture}[baseline = 12pt]
        %Draw bottom vertices
        \draw \foreach \m in {0,1} {
            (\m,0) node[circle, draw, fill=black,inner sep=0pt, minimum width=4pt]{}
            (\m,1) node[circle, draw, fill=black,inner sep=0pt, minimum width=4pt]{}
            };
        %Draw jellyfish
        \draw[-,thick] (0,0) to (0,1);
        \filldraw[fill=\jellycolor,-,thick] (0.8,0.35) to (1.8,0.35) to[out=up,in=right] (1.3,0.8) to[out=left,in=up] (0.8,0.35);
        \draw[-,thick] (1,0) to[out=up,in=down] (1.6,0.35);
        \draw[-,thick] (1,1) to[out=down,in=up] (0.4,0.35) to[out=down,in=down] (1,0.35);
    \end{tikzpicture}
\] 
Consider the following composition:
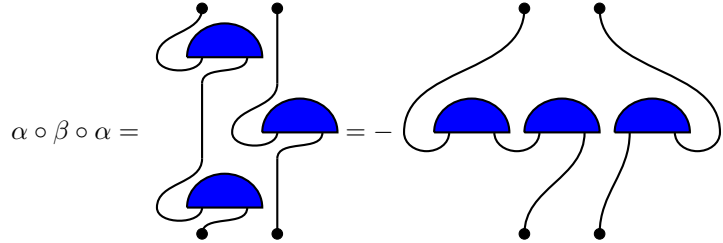
\[\alpha\circ \beta\circ \alpha=~
    \begin{tikzpicture}[baseline = 36pt]
    %Draw bottom vertices
        \draw \foreach \m in {1,...,2} {
            (\m,0) node[circle, draw, fill=black,inner sep=0pt, minimum width=4pt]{}
            (\m,3) node[circle, draw, fill=black,inner sep=0pt, minimum width=4pt]{}
            };
        %Draw jellyfish
        \draw[-,thick] (2,0) to (2,1);
        \filldraw[fill=\jellycolor,-,thick] (0.8,0.35) to (1.8,0.35) to[out=up,in=right] (1.3,0.8) to[out=left,in=up] (0.8,0.35);
        \draw[-,thick] (1,0) to[out=up,in=down] (1.6,0.35);
        \draw[-,thick] (1,1) to[out=down,in=up] (0.4,0.35) to[out=down,in=down] (1,0.35);
        \draw[-,thick] (2,2) to (2,3);
        \filldraw[fill=\jellycolor,-,thick] (0.8,2.35) to (1.8,2.35) to[out=up,in=right] (1.3,2.8) to[out=left,in=up] (0.8,2.35);
        \draw[-,thick] (1,2) to[out=up,in=down] (1.6,2.35);
        \draw[-,thick] (1,3) to[out=down,in=up] (0.4,2.35) to[out=down,in=down] (1,2.35);
        \draw[-,thick] (1,1) to (1,2);
        \filldraw[fill=\jellycolor,-,thick] (1.8,1.35) to (2.8,1.35) to[out=up,in=right] (2.3,1.8) to[out=left,in=up] (1.8,1.35);
        \draw[-,thick] (2,1) to[out=up,in=down] (2.6,1.35);
        \draw[-,thick] (2,2) to[out=down,in=up] (1.4,1.35) to[out=down,in=down] (2,1.35);
    \end{tikzpicture}
    =-~
    \begin{tikzpicture}[baseline = 36pt]
        %Draw bottom vertices
        \draw \foreach \m in {2,3} {
            %(\m,0) node[circle, draw, fill=black,inner sep=0pt, minimum width=4pt]{}
            (\m,0) node[circle, draw, fill=black,inner sep=0pt, minimum width=4pt]{}
            (\m,3) node[circle, draw, fill=black,inner sep=0pt, minimum width=4pt]{}
            %(\m,3) node[circle, draw, fill=black,inner sep=0pt, minimum width=4pt]{}
            };
        %Draw jellyfish
        \filldraw[fill=\jellycolor,-,thick] (0.8,1.35) to (1.8,1.35) to[out=up,in=right] (1.3,1.8) to[out=left,in=up] (0.8,1.35);
        \filldraw[fill=\jellycolor,-,thick] (2,1.35) to (3,1.35) to[out=up,in=right] (2.5,1.8) to[out=left,in=up] (2,1.35);
        \filldraw[fill=\jellycolor,-,thick] (3.2,1.35) to (4.2,1.35) to[out=up,in=right] (3.7,1.8) to[out=left,in=up] (3.2,1.35);
        \draw[-,thick] (2,3) to[out=down,in=up] (0.4,1.35) to[out=down,in=left] (0.7,1.1) to[out=right,in=down] (1,1.35);
        \draw[-,thick] (1.6,1.35) to[out=down,in=left] (1.9,1.1) to[out=right,in=down] (2.2,1.35);
        \draw[-,thick] (2.8,1.35) to[out=down,in=up] (2,0);
        \draw[-,thick] (3.4,1.35) to[out=down,in=up] (3,0);
        \draw[-,thick] (3,3) to[out=down,in=up] (4.6,1.35) to[out=down,in=right] (4.3,1.1) to[out=left,in=down] (4,1.35);
    \end{tikzpicture}
\]
We can reduce the number of jellyfish in the diagram above by using the jellyfish relation on any of the three pairs of jellyfish:
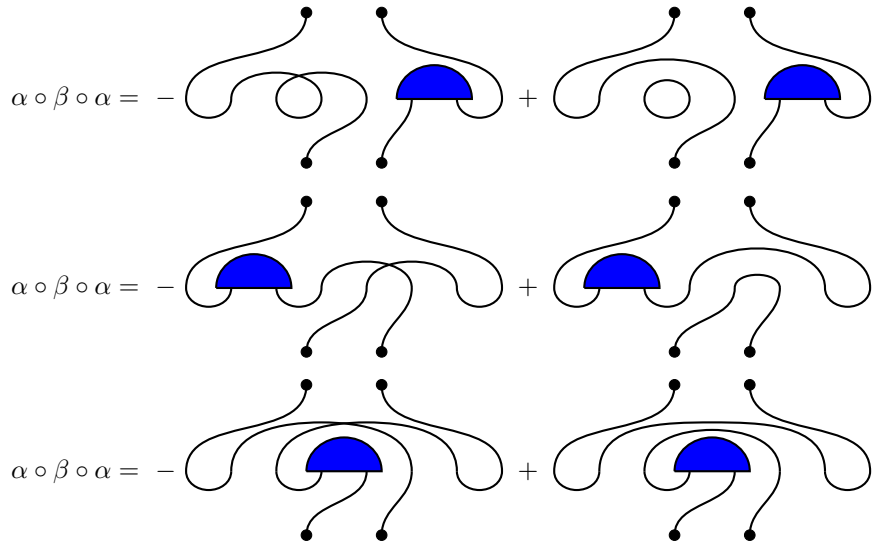
\[\alpha\circ \beta\circ \alpha=~-~
    \begin{tikzpicture}[baseline = 36pt]
        %Draw vertices
        \draw \foreach \m in {2,3} {
            (\m,0.5) node[circle, draw, fill=black,inner sep=0pt, minimum width=4pt]{}
            (\m,2.5) node[circle, draw, fill=black,inner sep=0pt, minimum width=4pt]{}
            };
        %Draw jellyfish
        \filldraw[fill=\jellycolor,-,thick] (3.2,1.35) to (4.2,1.35) to[out=up,in=right] (3.7,1.8) to[out=left,in=up] (3.2,1.35);
        \draw[-,thick] (2,2.5) to[out=down,in=up] (0.4,1.35) to[out=down,in=left] (0.7,1.1) to[out=right,in=down] (1,1.35);
        \draw[-,thick] (1.6,1.35) to[out=down,in=left] (1.9,1.1) to[out=right,in=down] (2.2,1.35);
        \draw[-,thick] (2.8,1.35) to[out=down,in=up] (2,0.5);
        \draw[-,thick] (3.4,1.35) to[out=down,in=up] (3,0.5);
        \draw[-,thick] (3,2.5) to[out=down,in=up] (4.6,1.35) to[out=down,in=right] (4.3,1.1) to[out=left,in=down] (4,1.35);
        \draw[-,thick] (1,1.35) to[out=up,in=up] (2.2,1.35);
        \draw[-,thick] (1.6,1.35) to[out=up,in=up] (2.8,1.35);
    \end{tikzpicture}
    ~+~
    \begin{tikzpicture}[baseline = 36pt]
        %Draw vertices
        \draw \foreach \m in {2,3} {
            (\m,0.5) node[circle, draw, fill=black,inner sep=0pt, minimum width=4pt]{}
            (\m,2.5) node[circle, draw, fill=black,inner sep=0pt, minimum width=4pt]{}
            };
        %Draw jellyfish
        \filldraw[fill=\jellycolor,-,thick] (3.2,1.35) to (4.2,1.35) to[out=up,in=right] (3.7,1.8) to[out=left,in=up] (3.2,1.35);
        \draw[-,thick] (2,2.5) to[out=down,in=up] (0.4,1.35) to[out=down,in=left] (0.7,1.1) to[out=right,in=down] (1,1.35);
        \draw[-,thick] (1.6,1.35) to[out=down,in=left] (1.9,1.1) to[out=right,in=down] (2.2,1.35);
        \draw[-,thick] (2.8,1.35) to[out=down,in=up] (2,0.5);
        \draw[-,thick] (3.4,1.35) to[out=down,in=up] (3,0.5);
        \draw[-,thick] (3,2.5) to[out=down,in=up] (4.6,1.35) to[out=down,in=right] (4.3,1.1) to[out=left,in=down] (4,1.35);
        \draw[-,thick] (1,1.35) to[out=up,in=up] (2.8,1.35);
        \draw[-,thick] (1.6,1.35) to[out=up,in=left] (1.9,1.6) to[out=right,in=up] (2.2,1.35);
    \end{tikzpicture}
\]
\[\alpha\circ \beta\circ \alpha=~-~
    \begin{tikzpicture}[baseline = 36pt]
        %Draw vertices
        \draw \foreach \m in {2,3} {
            (\m,0.5) node[circle, draw, fill=black,inner sep=0pt, minimum width=4pt]{}
            (\m,2.5) node[circle, draw, fill=black,inner sep=0pt, minimum width=4pt]{}
            };
        %Draw jellyfish
         \filldraw[fill=\jellycolor,-,thick] (0.8,1.35) to (1.8,1.35) to[out=up,in=right] (1.3,1.8) to[out=left,in=up] (0.8,1.35);
        \draw[-,thick] (2,2.5) to[out=down,in=up] (0.4,1.35) to[out=down,in=left] (0.7,1.1) to[out=right,in=down] (1,1.35);
        \draw[-,thick] (1.6,1.35) to[out=down,in=left] (1.9,1.1) to[out=right,in=down] (2.2,1.35);
        \draw[-,thick] (2.8,1.35) to[out=down,in=up] (2,0.5);
        \draw[-,thick] (3.4,1.35) to[out=down,in=up] (3,0.5);
        \draw[-,thick] (3,2.5) to[out=down,in=up] (4.6,1.35) to[out=down,in=right] (4.3,1.1) to[out=left,in=down] (4,1.35);
        \draw[-,thick] (2.2,1.35) to[out=up,in=up] (3.4,1.35);
        \draw[-,thick] (2.8,1.35) to[out=up,in=up] (4,1.35);
    \end{tikzpicture}
    ~+~
    \begin{tikzpicture}[baseline = 36pt]
        %Draw vertices
        \draw \foreach \m in {2,3} {
            (\m,0.5) node[circle, draw, fill=black,inner sep=0pt, minimum width=4pt]{}
            (\m,2.5) node[circle, draw, fill=black,inner sep=0pt, minimum width=4pt]{}
            };
        %Draw jellyfish
         \filldraw[fill=\jellycolor,-,thick] (0.8,1.35) to (1.8,1.35) to[out=up,in=right] (1.3,1.8) to[out=left,in=up] (0.8,1.35);
        \draw[-,thick] (2,2.5) to[out=down,in=up] (0.4,1.35) to[out=down,in=left] (0.7,1.1) to[out=right,in=down] (1,1.35);
        \draw[-,thick] (1.6,1.35) to[out=down,in=left] (1.9,1.1) to[out=right,in=down] (2.2,1.35);
        \draw[-,thick] (2.8,1.35) to[out=down,in=up] (2,0.5);
        \draw[-,thick] (3.4,1.35) to[out=down,in=up] (3,0.5);
        \draw[-,thick] (3,2.5) to[out=down,in=up] (4.6,1.35) to[out=down,in=right] (4.3,1.1) to[out=left,in=down] (4,1.35);
        \draw[-,thick] (2.2,1.35) to[out=up,in=up] (4,1.35);
        \draw[-,thick] (2.8,1.35) to[out=up,in=up] (3.4,1.35);
    \end{tikzpicture}
\]
\[\alpha\circ \beta\circ \alpha=~-~
    \begin{tikzpicture}[baseline = 36pt]
        %Draw vertices
        \draw \foreach \m in {2,3} {
            (\m,0.5) node[circle, draw, fill=black,inner sep=0pt, minimum width=4pt]{}
            (\m,2.5) node[circle, draw, fill=black,inner sep=0pt, minimum width=4pt]{}
            };
        %Draw jellyfish
        \filldraw[fill=\jellycolor,-,thick] (2,1.35) to (3,1.35) to[out=up,in=right] (2.5,1.8) to[out=left,in=up] (2,1.35);
        \draw[-,thick] (2,2.5) to[out=down,in=up] (0.4,1.35) to[out=down,in=left] (0.7,1.1) to[out=right,in=down] (1,1.35);
        \draw[-,thick] (1.6,1.35) to[out=down,in=left] (1.9,1.1) to[out=right,in=down] (2.2,1.35);
        \draw[-,thick] (2.8,1.35) to[out=down,in=up] (2,0.5);
        \draw[-,thick] (3.4,1.35) to[out=down,in=up] (3,0.5);
        \draw[-,thick] (3,2.5) to[out=down,in=up] (4.6,1.35) to[out=down,in=right] (4.3,1.1) to[out=left,in=down] (4,1.35);
        \draw[-,thick] (1,1.35) to[out=up,in=left] (2.2,2) to[out=right,in=up] (3.4,1.35);
        \draw[-,thick] (1.6,1.35) to[out=up,in=left] (2.8,2) to[out=right,in=up] (4,1.35);
    \end{tikzpicture}
    ~+~
    \begin{tikzpicture}[baseline = 36pt]
        %Draw vertices
        \draw \foreach \m in {2,3} {
            (\m,0.5) node[circle, draw, fill=black,inner sep=0pt, minimum width=4pt]{}
            (\m,2.5) node[circle, draw, fill=black,inner sep=0pt, minimum width=4pt]{}
            };
        %Draw jellyfish
        \filldraw[fill=\jellycolor,-,thick] (2,1.35) to (3,1.35) to[out=up,in=right] (2.5,1.8) to[out=left,in=up] (2,1.35);
        \draw[-,thick] (2,2.5) to[out=down,in=up] (0.4,1.35) to[out=down,in=left] (0.7,1.1) to[out=right,in=down] (1,1.35);
        \draw[-,thick] (1.6,1.35) to[out=down,in=left] (1.9,1.1) to[out=right,in=down] (2.2,1.35);
        \draw[-,thick] (2.8,1.35) to[out=down,in=up] (2,0.5);
        \draw[-,thick] (3.4,1.35) to[out=down,in=up] (3,0.5);
        \draw[-,thick] (3,2.5) to[out=down,in=up] (4.6,1.35) to[out=down,in=right] (4.3,1.1) to[out=left,in=down] (4,1.35);
        \draw[-,thick] (1,1.35) to[out=up,in=left] (2.5,2) to[out=right,in=up] (4,1.35);
        \draw[-,thick] (1.6,1.35) to[out=up,in=left] (2.5,1.9) to[out=right,in=up] (3.4,1.35);
    \end{tikzpicture}
\]
Simplifying the thee expressions for $\alpha\circ \beta\circ \alpha$ above we conclude:
\begin{equation}\label{Grood example}
\begin{tikzpicture}[baseline = 12pt]
    %Draw bottom vertices
    \draw \foreach \m in {0,1} {
        (\m,0) node[circle, draw, fill=black,inner sep=0pt, minimum width=4pt]{}
        (\m,1) node[circle, draw, fill=black,inner sep=0pt, minimum width=4pt]{}
        };
        %Draw jellyfish
        \draw[-,thick] (0,0) to (0,1);
        \filldraw[fill=\jellycolor,-,thick] (0.2,0.35) to (1.2,0.35) to[out=up,in=right] (0.7,0.8) to[out=left,in=up] (0.2,0.35);
        \draw[-,thick] (1,0) to[out=up,in=down] (0.4,0.35);
        \draw[-,thick] (1,1) to[out=down,in=up] (1.6,0.35) to[out=down,in=down] (1,0.35);
\end{tikzpicture}
~=~
\begin{tikzpicture}[baseline = 12pt]
    %Draw bottom vertices
    \draw \foreach \m in {0,1} {
        (\m,0) node[circle, draw, fill=black,inner sep=0pt, minimum width=4pt]{}
        (\m,1) node[circle, draw, fill=black,inner sep=0pt, minimum width=4pt]{}
        };
        %Draw jellyfish
        \draw[-,thick] (0,0) to[out=up,in=up] (1,0);
        \filldraw[fill=\jellycolor,-,thick] (0,0.5) to (1,0.5) to[out=up,in=right] (0.5,0.95) to[out=left,in=up] (0,0.5);
        \draw[-,thick] (1,1) to[out=down,in=up] (1.4,0.5) to[out=down,in=down] (0.8,0.5);
        \draw[-,thick] (0,1) to[out=down,in=up] (-0.4,0.5) to[out=down,in=down] (0.2,0.5);
\end{tikzpicture}
~-~
\begin{tikzpicture}[baseline = 12pt]
    %Draw bottom vertices
    \draw \foreach \m in {0,1} {
        (\m,0) node[circle, draw, fill=black,inner sep=0pt, minimum width=4pt]{}
        (\m,1) node[circle, draw, fill=black,inner sep=0pt, minimum width=4pt]{}
        };
        %Draw jellyfish
        \draw[-,thick] (1,0) to[out=up,in=down] (0.8,0.5);
        \filldraw[fill=\jellycolor,-,thick] (0,0.5) to (1,0.5) to[out=up,in=right] (0.5,0.95) to[out=left,in=up] (0,0.5);
        \draw[-,thick] (1,1) to[out=down,in=up] (1.2,0.5) to[out=down,in=up] (0,0);
        \draw[-,thick] (0,1) to[out=down,in=up] (-0.4,0.5) to[out=down,in=down] (0.2,0.5);
\end{tikzpicture}
~=~
\begin{tikzpicture}[baseline = 12pt]
    %Draw bottom vertices
    \draw \foreach \m in {0,1} {
        (\m,0) node[circle, draw, fill=black,inner sep=0pt, minimum width=4pt]{}
        (\m,1) node[circle, draw, fill=black,inner sep=0pt, minimum width=4pt]{}
        };
        %Draw jellyfish
        \draw[-,thick] (0,0) to[out=up,in=down] (0.2,0.5);
        \filldraw[fill=\jellycolor,-,thick] (0,0.5) to (1,0.5) to[out=up,in=right] (0.5,0.95) to[out=left,in=up] (0,0.5);
        \draw[-,thick] (1,1) to[out=down,in=up] (1.4,0.5) to[out=down,in=down] (0.8,0.5);
        \draw[-,thick] (0,1) to[out=down,in=up] (-0.2,0.5) to[out=down,in=up] (1,0);
\end{tikzpicture}
~-~
\begin{tikzpicture}[baseline = 12pt]
    %Draw bottom vertices
    \draw \foreach \m in {0,1} {
        (\m,0) node[circle, draw, fill=black,inner sep=0pt, minimum width=4pt]{}
        (\m,1) node[circle, draw, fill=black,inner sep=0pt, minimum width=4pt]{}
        };
        %Draw jellyfish
        \draw[-,thick] (0,1) to[out=-45,in=-135] (1,1);
        \filldraw[fill=\jellycolor,-,thick] (0,0.25) to (1,0.25) to[out=up,in=right] (0.5,0.7) to[out=left,in=up] (0,0.25);
        \draw[-,thick] (0,0) to[out=up,in=down] (0.2,0.25);
        \draw[-,thick] (1,0) to[out=up,in=down] (0.8,0.25);
\end{tikzpicture}
\end{equation}
\end{example}

\begin{remark}\label{Grood remark} 
Let us compare the previous example with \cite[Example 3.1]{Grood}, where it is shown that even Brauer algebras are not associative. Grood sets
\[a=~\begin{tikzpicture}[baseline = 8pt, scale=.75]
        \draw[-,thick] (1,1) to (1,0);
        %Draw vertices
        \draw \foreach \m in {0,1} {
            (\m,0) node[circle, draw, fill=white, thick, inner sep=0pt, minimum width=4pt]{}
            (\m,1) node[circle, draw, fill=white, thick, inner sep=0pt, minimum width=4pt]{}
            };
    \end{tikzpicture}
    \quad\text{and}\quad b=~
    \begin{tikzpicture}[baseline = 8pt, scale=.75]
        \draw[-,thick] (0,1) to (0,0);
        %Draw vertices
        \draw \foreach \m in {0,1} {
            (\m,0) node[circle, draw, fill=white, thick, inner sep=0pt, minimum width=4pt]{}
            (\m,1) node[circle, draw, fill=white, thick, inner sep=0pt, minimum width=4pt]{}
            };
    \end{tikzpicture}
\]
Note that adding jellyfish to $a$ and $b$ in the manner prescribed in \S\ref{jellyfish Brauer} results in $\alpha$ and $\beta$ respectively. Grood computes\footnote{Grood uses the convention that $ab$ is obtained by stacking $a$ below $b$.} 
\[(ab)a=~\begin{tikzpicture}[baseline = 8pt, scale=0.75]
        \draw[-,thick] (0,0) to[out=up,in=up] (1,0);
        %Draw vertices
        \draw \foreach \m in {0,1} {
            (\m,0) node[circle, draw, fill=white, thick, inner sep=0pt, minimum width=4pt]{}
            (\m,1) node[circle, draw, fill=white, thick, inner sep=0pt, minimum width=4pt]{}
            };
    \end{tikzpicture}
    ~-~
    \begin{tikzpicture}[baseline = 8pt, scale=0.75]
        \draw[-,thick] (0,0) to[out=up,in=down] (1,1);
        %Draw vertices
        \draw \foreach \m in {0,1} {
            (\m,0) node[circle, draw, fill=white, thick, inner sep=0pt, minimum width=4pt]{}
            (\m,1) node[circle, draw, fill=white, thick, inner sep=0pt, minimum width=4pt]{}
            };
    \end{tikzpicture}
    \quad\text{and}\quad
    a(ba)=~-~\begin{tikzpicture}[baseline = 8pt, scale=0.75]
        \draw[-,thick] (0,1) to[out=down,in=up] (1,0);
        %Draw vertices
        \draw \foreach \m in {0,1} {
            (\m,0) node[circle, draw, fill=white, thick, inner sep=0pt, minimum width=4pt]{}
            (\m,1) node[circle, draw, fill=white, thick, inner sep=0pt, minimum width=4pt]{}
            };
    \end{tikzpicture}
    ~-~\begin{tikzpicture}[baseline = 8pt, scale=0.75]
        \draw[-,thick] (0,1) to[out=down,in=down] (1,1);
        %Draw vertices
        \draw \foreach \m in {0,1} {
            (\m,0) node[circle, draw, fill=white, thick, inner sep=0pt, minimum width=4pt]{}
            (\m,1) node[circle, draw, fill=white, thick, inner sep=0pt, minimum width=4pt]{}
            };
    \end{tikzpicture}
\]
Adding jellyfish to the diagrams above and uncrossing a pair of legs results in the right two expressions for $\alpha\circ\beta\circ\alpha$ in (\ref{Grood example}).  
\end{remark}

\section{Orbits for $A_n$ and the fullness of $\Psi$}\label{Psi}

In this section we use jellyfish diagrams to classify the $A_n$-orbits of our basis for $V^{\otimes k}$. As a consequence of the manner in which those orbits are classified, we can prove that $\Psi$ is full (Theorem \ref{Psi is full}). We also get a description of the dimensions of the spaces $\Hom_{A_n}(V^{\otimes k},V^{\otimes \ell})$ in terms of Bell and Stirling numbers in \S\ref{counting dimensions}.

\subsection{$A_n$-orbits}\label{subsection: A-orbits} 
Let $D:k\to0$ denote a partition diagram with at most $n$ parts and recall the definition of the associated $S_n$-orbit $O_D$ from \S\ref{repn Sn}. Suppose for a moment that $D$ has at most $n-2$ parts. Then for any $v_{\bd i}\in D$, there exist distinct $a,b\in\{1,\ldots,n\}$ such that the transposition $(a,b)$ fixes $v_{\bd i}$. In particular, $\sigma(a,b)\cdot v_{\bd i}=\sigma\cdot v_{\bd i}$ for every $\sigma\in S_n$. It follows that the $S_n$-orbit $O_D$ is also the $A_n$-orbit of $v_{\bd i}$. On the other hand, if $D$ has $n-1$ or $n$ parts, then the action of $S_n$ on $O_D$ is faithful. Indeed, this follows from the fact that a permutation of $\{1,\ldots,n\}$ is completely determined by its action on $n-1$ of the elements in $\{1,\ldots,n\}$. Since $A_n$ is a subgroup of $S_n$ of index 2, it follows that the $S_n$-orbit $O_D$ is the disjoint union of two $A_n$-orbits whenever $D$ has $n-1$ or $n$ parts. We will now use jellyfish to distinguish between those two $A_n$-orbits.

To each partition diagram $D:k\to 0$ with $n$ or $n-1$ parts we associate a jellyfish diagram $\jelly_D:k\to0$ in $\JP(n)$ as follows. First, mark the leftmost vertex in each part of $D$. Now place a jellyfish above $D$. In the case that $D$ has $n-1$ parts make the rightmost leg of the jellyfish dangling. Attach all of the other jellyfish legs to the marked vertices of $D$ in such a way that no two legs are crossing. For example, suppose  
\[
    D=
    \begin{tikzpicture}[baseline = 2pt, scale=0.75]
         %Draw bottom vertices
        \draw \foreach \m in {0,...,8} {
        (0.5*\m,0) node[circle, draw, fill=black]{}
        };
        %Draw edges:
        \draw[-,thick] (0,0) to[out=up,in=up] (1.5,0) to[out=up,in=up] (2.5,0);
        \draw[-,thick] (1,0) to[out=up,in=up] (3,0) to[out=up,in=up] (4,0);
        \draw[-,thick] (2,0) to[out=up,in=up] (3.5,0);
    \end{tikzpicture}
    \quad\text{and}\quad
    D'=
    \begin{tikzpicture}[baseline = 2pt, scale=0.75]
         %Draw bottom vertices
        \draw \foreach \m in {0,...,7} {
        (0.5*\m,0) node[circle, draw, fill=black]{}
        };
        %Draw edges:
        \draw[-,thick] (0,0) to[out=up,in=up] (1.5,0);
        \draw[-,thick] (0.5,0) to[out=up,in=up] (1,0) to[out=up,in=up] (2.5,0) to[out=up,in=up] (3,0);
        \draw[-,thick] (2,0) to[out=up,in=up] (3.5,0);
    \end{tikzpicture}
\]
If $n=4$, then the corresponding jellyfish diagrams are 
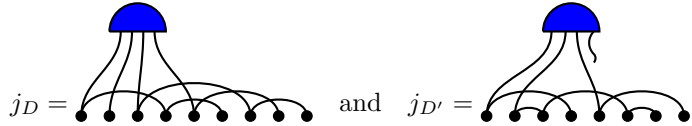
\[
    \jelly_D=
    \begin{tikzpicture}[baseline = 2pt, scale=0.75]
         %Draw bottom vertices
        \draw \foreach \m in {0,...,8} {
        (0.5*\m,0) node[circle, draw, fill=black]{}
        };
        %Draw edges:
        \draw[-,thick] (0,0) to[out=up,in=up] (1.5,0) to[out=up,in=up] (2.5,0);
        \draw[-,thick] (1,0) to[out=up,in=up] (3,0) to[out=up,in=up] (4,0);
        \draw[-,thick] (2,0) to[out=up,in=up] (3.5,0);
        \draw[-,thick] (1.3, 1.5) to[out=down,in=up] (2,0);
        \draw[-,thick] (1.1,1.5) to[out=down,in=up] (1,0);
        \draw[-,thick] (0.9,1.5) to[out=down,in=up] (0.5,0);
        \draw[-,thick] (0.7,1.5) to[out=down,in=up] (0,0);
        %Draw jellyfish
        \filldraw[fill=\jellycolor,-,thick] (0.5,1.5) to (1.5,1.5) to[out=up, in=right] (1,2) to[out=left, in=up] (0.5,1.5);
    \end{tikzpicture}
    \quad\text{and}\quad
    \jelly_{D'}=
    \begin{tikzpicture}[baseline = 2pt, scale=0.75]
         %Draw bottom vertices
        \draw \foreach \m in {0,...,7} {
        (0.5*\m,0) node[circle, draw, fill=black]{}
        };
        %Draw edges:
        \draw[-,thick] (0,0) to[out=up,in=up] (1.5,0);
        \draw[-,thick] (0.5,0) to[out=up,in=up] (1,0) to[out=up,in=up] (2.5,0) to[out=up,in=up] (3,0);
        \draw[-,thick] (2,0) to[out=up,in=up] (3.5,0);
        \draw[-,thick] (1.9, 1.5) to[out=-135,in=135] (1.9,1.1) to[out=-45,in=45] (1.9,0.95);
        \draw[-,thick] (1.6,1.5) to[out=down,in=up] (2,0);
        \draw[-,thick] (1.4,1.5) to[out=down,in=up] (0.5,0);
        \draw[-,thick] (1.15,1.5) to[out=down,in=up] (0,0);
        %Draw jellyfish
        \filldraw[fill=\jellycolor,-,thick] (1,1.5) to (2,1.5) to[out=up, in=right] (1.5,2) to[out=left, in=up] (1,1.5);
    \end{tikzpicture}
\]

\begin{proposition}\label{jelly-D} 
Suppose $D:k\to 0$ is a partition diagram with $n$ or $n-1$ parts and fix a basis vector $v_{\bd i}$ in $V^{\otimes k}$.
    
    (1) $\Psi(\jelly_D)(\sigma\cdot v_{\bd i})=(-1)^\sigma\Psi(\jelly_D)(v_{\bd i})$ for any $\sigma\in S_n$.
    
    (2)
    $\Psi(\jelly_D):v_{\bd i}\mapsto
    \begin{cases}
        \pm 1, & \text{if }v_{\bd i}\in O_D;\\
        0, & \text{otherwise}.
    \end{cases}$
\end{proposition}

\begin{proof}
    Let $\hat{D}$ denote the partition diagram of type $k\to n$ obtained from $\jelly_D$ by cutting the jellyfish legs just below the body. Then we have $\Psi(\jelly_D)=\det\circ\Phi(\hat{D})$.  Now (1) follows from the fact that $\Phi(\hat{D})$ commutes with the action of $S_n$ along with (\ref{det property}). 

    To prove (2) let $D':k\to0$ denote the partition diagram with $v_{\bd i}\in O_{D'}$. If $D'\not\geq D$, then $\Phi(\hat{D})(v_{\bd i})=0$. If $D'\gneqq D$, then the terms appearing in $\Phi(\hat{D})(v_{\bd i})$ will all have the form $v_{\bd i'}$ with at least two of $i'_1,\ldots,i'_n$ equal, and $\det(v_{\bd i'})=0$ for all such ${\bd i'}$. Finally, if $D'=D$ then there will be a unique term of $\Phi(\hat{D})(v_{\bd i})$ of the form $v_{\bd i'}$ with $i'_1,\ldots,i'_n$ pairwise distinct, and $\det(v_{\bd i'})=\pm 1$ for such an ${\bd i'}$.  
\end{proof}

Given a partition diagram $D:k\to 0$ having $n$ or $n-1$ parts we set 
\[
    O_D^+=\{v_{\bd i}~|~j_D(v_{\bd i})=1\},
    \quad
    O_D^-=\{v_{\bd i}~|~j_D(v_{\bd i})=-1\}.
\]
It follows from part (2) of Proposition \ref{jelly-D} that $O_D=O_D^+\cup O_D^-$. Moreover, it follows from part (1) of Proposition \ref{jelly-D} that $O_D^+$ and $O_D^-$ are precisely the two $A_n$-orbits inside $O_D$. Let $f_D^\pm:V^{\otimes k}\to\unit$ denote the $\k$-linear maps defined by 
\[
    f_D^+(v_{\bd i})=
    \begin{cases}
        1, & \text{if }v_{\bd i}\in O_D^+;\\
        0, & \text{if }v_{\bd i}\not\in O_D^+.
    \end{cases}
    \qquad
    f_D^-(v_{\bd i})=
    \begin{cases}
        1, & \text{if }v_{\bd i}\in O_D^-;\\
        0, & \text{if }v_{\bd i}\not\in O_D^-.
    \end{cases}
\]
We get the following $A_n$-analog of Proposition \ref{prop: f basis}:

\begin{proposition}\label{prop: f pm basis}
    The set $\{f_D\}_D\cup\{f^\pm_D\}_D$ is a basis for $\Hom_{A_n}(V^{\otimes k},\unit)$ where the first (resp.~second) set is indexed by all partition diagrams $D:k\to 0$ having at most $n-2$ parts (resp.~ having $n$ or $n-1$ parts).
\end{proposition}

\subsection{Proof of Theorem \ref{Psi is full}}\label{proof of theorem A}
To prove that $\Psi$ is full, consider the following commutative diagram:
\begin{equation}\label{Reduce to invariants alt}
    \begin{CD}
    \Hom_{\JP(n)}(k,\ell)&@>>>&
    \Hom_{\JP(n)}(k+\ell,0)\\
    @V \Psi VV&
    &@V \Psi VV\\
    \Hom_{A_n}(V^{\otimes k}, V^{\otimes \ell})&@>>>&
    \Hom_{A_n}(V^{\otimes k+\ell}, \unit).
    \end{CD}
\end{equation}
The horizontal maps in (\ref{Reduce to invariants alt}) are the $\k$-linear isomorphisms defined in the exact same way as their counterparts in (\ref{Reduce to invariants symm}).  Since (\ref{Reduce to invariants alt}) commutes, it suffices to prove $\Hom_{\JP(n)}(k,0)\arup{\Psi}\Hom_{A_n}(V^{\otimes k},\unit)$ is surjective for each $k$. Suppose $D:k\to0$ is a partition diagram. It follows from (\ref{x to f}) and the definition of $\Psi$ that $\Psi(x_D)=f_D$. Moreover, if $D$ has $n$ or $n-1$ parts, it follows from part (2) of Proposition \ref{jelly-D} that $\Psi:\frac{1}{2}(x_D\pm \jelly_D)\mapsto f_D^\pm$. Hence, we are done by Proposition \ref{prop: f pm basis}. \qed

\subsection{Dimension formulae}\label{counting dimensions} The so-called \emph{Stirling number of the second kind}, denoted $\Stirling(k,p)$, is the number of partitions of a set with $k$ elements into $p$ mutually disjoint nonempty subsets. Note that $\Stirling(k,p)=0$ whenever $k<p$, $\Stirling(k,0)=0$ whenever $k>0$, and (by convention) $\Stirling(0,0)=1$. The \emph{Bell number} $\Bell(k)=\sum_{p=0}^k\Stirling(k,p)$ is the total number of set partitions of a set with $k$ elements.  In particular, by the definition of $\P(n)$ we have 
\begin{equation*}
    \dim_\k\Hom_{\P(n)}(k,\ell)=\Bell(k+\ell).
\end{equation*}
It follows from the bottom isomorphism in (\ref{Reduce to invariants symm}) along with Proposition \ref{prop: f basis} that 
\begin{equation*}
    \dim_\k\Hom_{S_n}(V^{\otimes k},V^{\otimes \ell})=\sum\limits_{p=0}^{n}\Stirling(k+\ell,p).
\end{equation*}
In particular, $\dim_\k\Hom_{S_n}(V^{\otimes k},V^{\otimes \ell})=\Bell(k+\ell)$ if and only if $k+\ell\leq n$.
Similarly, it follows from the bottom isomorphism in (\ref{Reduce to invariants alt}) along with Proposition \ref{prop: f pm basis} that 
\begin{equation}\label{dimensions for A}
\dim_\k\Hom_{A_n}(V^{\otimes k},V^{\otimes \ell})=\sum\limits_{p=0}^{n-2}\Stirling(k+\ell,p)+2\Stirling(k+\ell,n-1)+2\Stirling(k+\ell,n).
\end{equation}
In particular, 
\begin{equation*}\label{special case dimensions}
    \dim_\k\Hom_{A_n}(V^{\otimes k},V^{\otimes \ell})=
    \begin{cases}
        \Bell(k+\ell), & \text{if }k+\ell\leq n-2;\\
        \Bell(n-1)+1, & \text{if }k+\ell=n-1;\\
        \Bell(n)+\Stirling(n,n-1)+1, & \text{if }k+\ell=n.\\
    \end{cases}
\end{equation*}
In the special case when $k=\ell$ the dimension formulae above can be found in \cite{Bloss}.

By Theorem \ref{Psi is full} we know $\dim_\k\Hom_{\JP(n)}(k,\ell)\geq\dim_\k\Hom_{A_n}(V^{\otimes k}, V^{\otimes \ell})$. 
In \S\ref{section: Psi is faithful} we will show that under certain constraints on $\operatorname{char}\k$ the reverse inequality also holds. We will do so by showing morphisms in $\JP(n)$ are spanned by a set of size (\ref{dimensions for A}).  In particular, this will imply that $\Psi$ is also faithful.

\section{The faithfulness of $\Psi$}\label{section: Psi is faithful}

From now on we assume $n>1$ and that $\operatorname{char} \k=0$ or $\operatorname{char}\k\geq n$. In this section we first prove three lemmas which result in a description of a basis for $\Hom_{\JP(n)}(k,0)$ (see Proposition \ref{basis for JP}). As a consequence we obtain a proof of Theorem \ref{Psi is faithful} in \S\ref{proof of main result}.

\subsection{Spanning sets for jellyfish diagrams}\label{subsection: Spanning sets for jellyfish diagrams}

For this subsection let $Y:n+1\to 0$ denote the unique partition diagram with $n+1$ parts. In other words, 
\[
    Y=~~
    \begin{tikzpicture}[baseline = 2pt, scale=0.75]
         %Draw bottom vertices
        \draw \foreach \m in {0,1,2,4} {
        (\m,0) node[circle, draw, fill=black]{}};
        %Labels
        \draw (3,0) node{$\cdots$};
        \draw (0,-0.4) node{1};
        \draw (1,-0.4) node{2};
        \draw (2,-0.4) node{3};
        \draw (4,-0.4) node{$n+1$};
    \end{tikzpicture}
\]
The following lemma and its proof generalize Example \ref{example p dependent}, the $n=2$ case.

\begin{lemma}\label{reduce Y}
The partition diagram 
$Y$ defined above is equal to a linear combination of partition diagrams each having at most $n$ parts in $\JP(n)$.  
\end{lemma}

\begin{proof} Consider the following jellyfish diagram:
\begin{equation}\label{4 jellies}
    \begin{tikzpicture}[baseline = 12pt]
        %Draw bottom vertices
        \draw[-,thick]  \foreach \m in {0.25,1.75,2.75,4.25} {
            (\m,0-0.8) node[circle, draw, fill=black,inner sep=0pt, minimum width=4pt]{}
            };
        %Draw jellyfish bodies
        \filldraw[fill=\jellycolor,-,thick] (0,0.75-0.8) to (2,0.75-0.8) to[out=up, in=right] (1,1.4-0.8) to[out=left, in=up] (0,0.75-0.8);
        \filldraw[fill=\jellycolor,-,thick] (2.5,0.75-0.8) to (4.5,0.75-0.8) to[out=up, in=right] (3.5,1.4-0.8) to[out=left, in=up] (2.5,0.75-0.8);
        \filldraw[fill=\jellycolor,-,thick] (0,2.15-0.8) to (2,2.15-0.8) to[out=down, in=right] (1,1.5-0.8) to[out=left, in=down] (0,2.15-0.8);
        \filldraw[fill=\jellycolor,-,thick] (2.5,2.15-0.8) to (4.5,2.15-0.8) to[out=down, in=right] (3.5,1.5-0.8) to[out=left, in=down] (2.5,2.15-0.8);
        %Draw jellyfish legs
        \draw[-,thick] (0.25,0-0.8) to (0.25,0.75-0.8);
        \draw[-,thick] (4.25,0-0.8) to (4.25,0.75-0.8);
        \draw[-,thick] (0.75,0.75-0.8) to (1.75,0-0.8) to (2.75,0.75-0.8);
        \draw[-,thick] (1.75,0.75-0.8) to (2.75,0-0.8) to (3.75,0.75-0.8);
        
        \draw[-,thick] (0.25,1.35) to[out=up,in=up] (3.25,1.35);
        \draw[-,thick] (1.25,1.35) to[out=up,in=up] (4.25,1.35);
        \draw[-,thick] (1.75, 1.35) to[out=45,in=-45] (1.75,1.55) to[out=135,in=-135] (1.75,1.75);
        \draw[-,thick] (2.75, 1.35) to[out=45,in=-45] (2.75,1.55) to[out=135,in=-135] (2.75,1.75);
        %Labels
        \draw (0.75,1.45) node{$\cdots$};
        \draw (3.75,1.45) node{$\cdots$};

        \draw (2.25,0-0.8) node{$\cdots$};
        \draw (1.35,0.65-0.8) node{$\cdots$};
        \draw (3.15,0.65-0.8) node{$\cdots$};
        \draw (0.25,-0.3-0.8) node{$1$};
        \draw (1.75,-0.3-0.8) node{$2$};
        \draw (2.75,-0.33-0.8) node{$n$};
        \draw (4.25,-0.33-0.8) node{$n+1$};
    \end{tikzpicture}
\end{equation}
Applying the jellyfish relation to the left and right pair of jellyfish respectively results in the following:
\[
    \sum_{\sigma,\tau\in S_n}(-1)^\sigma(-1)^\tau~
    \begin{tikzpicture}[baseline = 12pt]
        %Draw bottom vertices
        \draw[-,thick]  \foreach \m in {0.25,1.75,2.75,4.25} {
            (\m,0-0.6) node[circle, draw, fill=black,inner sep=0pt, minimum width=4pt]{}
            };
        %Draw permutation boxes
        \draw[-,thick] (0,0.75-0.6) to (2,0.75-0.6) to (2, 1.4-0.6) to  (0,1.4-0.6) to (0,0.75-0.6);
        \draw (1,1.075-0.6) node{$\sigma$};
        \draw[-,thick] (2.5,0.75-0.6) to (4.5,0.75-0.6) to (4.5, 1.4-0.6) to  (2.5,1.4-0.6) to (2.5,0.75-0.6);
        \draw (3.5,1.075-0.6) node{$\tau$};
        %Draw edges
        \draw[-,thick] (0.25,0-0.6) to (0.25,0.75-0.6);
        \draw[-,thick] (4.25,0-0.6) to (4.25,0.75-0.6);
        \draw[-,thick] (0.75,0.75-0.6) to (1.75,0-0.6) to (2.75,0.75-0.6);
        \draw[-,thick] (1.75,0.75-0.6) to (2.75,0-0.6) to (3.75,0.75-0.6);
        
        \draw[-,thick] (0.25, 0.8) to[out=45,in=-45] (0.25,1) to[out=135,in=-135] (0.25,1.2);
        \draw[-,thick] (4.25, 0.8) to[out=45,in=-45] (4.25,1) to[out=135,in=-135] (4.25,1.2);
        \draw[-,thick] (0.75, 0.8) to[out=up, in=up] (2.75,0.8);
        \draw[-,thick] (1.75, 0.8) to[out=up, in=up] (3.75,0.8);
        %Labels
        \draw (1.3,0.9) node{$\cdots$};
        \draw (3.25,0.9) node{$\cdots$};

        \draw (2.25,0-0.6) node{$\cdots$};
        \draw (1.35,0.65-0.6) node{$\cdots$};
        \draw (3.15,0.65-0.6) node{$\cdots$};
        \draw (0.25,-0.3-0.6) node{$1$};
        \draw (1.75,-0.3-0.6) node{$2$};
        \draw (2.75,-0.33-0.6) node{$n$};
        \draw (4.25,-0.33-0.6) node{$n+1$};
    \end{tikzpicture}
\]
The terms in the sum above with partition diagram having more than $n$ parts are precisely 
\[
    \sum_{\sigma\in S_{n-1}}~
    \begin{tikzpicture}[baseline = 12pt]
        %Draw bottom vertices
        \draw[-,thick]  \foreach \m in {0.25,1.75,2.75,4.25} {
            (\m,0-0.6) node[circle, draw, fill=black,inner sep=0pt, minimum width=4pt]{}
            };
        %Draw permutation boxes
        \draw[-,thick] (0.5,0.75-0.6) to (2,0.75-0.6) to (2, 1.4-0.6) to  (0.5,1.4-0.6) to (0.5,0.75-0.6);
        \draw (1.25,1.075-0.6) node{$\sigma$};
        \draw[-,thick] (2.5,0.75-0.6) to (4,0.75-0.6) to (4, 1.4-0.6) to  (2.5,1.4-0.6) to (2.5,0.75-0.6);
        \draw (3.25,1.075-0.6) node{$\sigma$};
        %Draw edges
        \draw[-,thick] (0.75,1.4-0.6) to[out=up,in=up] (2.75,1.4-0.6);
        \draw[-,thick] (1.75,1.4-0.6) to[out=up,in=up] (3.75,1.4-0.6);
        \draw[-,thick] (0.25,0-0.6) to (0.25,1.4-0.6);
        \draw[-,thick] (4.25,0-0.6) to (4.25,1.4-0.6);
        \draw[-,thick] (0.75,0.75-0.6) to (1.75,0-0.6) to (2.75,0.75-0.6);
        \draw[-,thick] (1.75,0.75-0.6) to (2.75,0-0.6) to (3.75,0.75-0.6);

        \draw[-,thick] (0.25, 1.4-0.6) to[out=45,in=-45] (0.25,1.6-0.6) to[out=135,in=-135] (0.25,1.8-0.6);
        \draw[-,thick] (4.25, 1.4-0.6) to[out=45,in=-45] (4.25,1.6-0.6) to[out=135,in=-135] (4.25,1.8-0.6);
        %Labels
        \draw (2.25,0-0.6) node{$\cdots$};
        \draw (1.35,0.65-0.6) node{$\cdots$};
        \draw (3.15,0.65-0.6) node{$\cdots$};
        \draw (1.25,1.5-0.6) node{$\cdots$};
        \draw (3.25,1.5-0.6) node{$\cdots$};
        \draw (0.25,-0.3-0.6) node{$1$};
        \draw (1.75,-0.3-0.6) node{$2$};
        \draw (2.75,-0.33-0.6) node{$n$};
        \draw (4.25,-0.33-0.6) node{$n+1$};
    \end{tikzpicture}
    =(n-1)! Y.
\]
On the other hand, applying the jellyfish relation to the top and bottom pair of jellyfish in (\ref{4 jellies}) respectively results only in partition diagrams having at most $n$ parts. Since $(n-1)!$ is not divisible by $\operatorname{char}\k$, we can solve for $Y$ in terms of partition diagrams each having at most $n$ parts.  
\end{proof}

\begin{lemma}\label{span partition diagrams}
In $\JP(n)$ every partition diagram is equal to a linear combination of partition diagrams each having at most $n$ parts.
\end{lemma}

\begin{proof}
    Since the assignments (\ref{cap it}) and (\ref{uncap it}) preserve the number of parts in a partition diagram, and prescribe a bijection $\Hom_{\JP(n)}(k,\ell)\to\Hom_{\JP(n)}(k+\ell,0)$, it suffices to prove the lemma for partition diagrams of type $k\to 0$. Suppose $D:k\to 0$ has $p>n$ parts. We will show that $D$ is equal to a linear combination of partition diagrams with fewer than $p$ parts, and the lemma will follow by induction. Let $\overline{D}:k\to n+1$ denote any partition diagram obtained from $D$ by adding a row of $n+1$ vertices above $D$ and connecting each of those vertices to a different part of $D$ (which can be done since $p>n$). For example, if $n=3$ and 
    \[
        D=
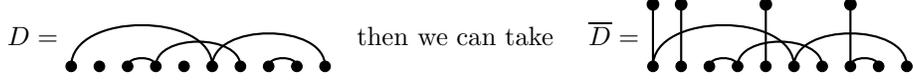
        \begin{tikzpicture}[baseline = 8pt, scale=0.75]
             %Draw bottom vertices
            \draw \foreach \m in {0,...,9} {
            (0.5*\m,0) node[circle, draw, fill=black]{}
            };
            %Draw edges:
            \draw[-,thick] (0,0) to[out=up,in=up] (2.5,0) to[out=up,in=up] (4.5,0);
            \draw[-,thick] (1,0) to[out=up,in=up] (1.5,0) to[out=up,in=up] (3,0);
            \draw[-,thick] (3.5,0) to[out=up,in=up] (4,0);
        \end{tikzpicture}
        \quad\text{then we can take }\quad
        \overline{D}=
        \begin{tikzpicture}[baseline = 8pt, scale=0.75]
             %Draw bottom vertices
            \draw \foreach \m in {0,...,9} {
            (0.5*\m,0) node[circle, draw, fill=black]{}
            };
            %Draw top vertices
            \draw[-,thick] \foreach \m in {0,0.5,2,3.5} {
            (\m,1.1) node[circle, draw, fill=black]{}
            (\m,0) to (\m,1.1)
            };
            %Draw edges:
            \draw[-,thick] (0,0) to[out=up,in=up] (2.5,0) to[out=up,in=up] (4.5,0);
            \draw[-,thick] (1,0) to[out=up,in=up] (1.5,0) to[out=up,in=up] (3,0);
            \draw[-,thick] (3.5,0) to[out=up,in=up] (4,0);
        \end{tikzpicture}
    \]
    Note that by construction $Z\circ \overline{D}$ is a partition diagram with fewer than $p$ parts whenever $Z:n+1\to0$ is a partition diagram with fewer than $n+1$ parts. Thus, it follows from Lemma \ref{reduce Y} that $D=Y\circ \overline{D}$ is a linear combination of partition diagrams with fewer than $p$ parts.
\end{proof}

\begin{lemma}\label{jelly-D span} 
Suppose $D:k\to n$ is a partition diagram with at most $n$ parts. Then, in $\JP(n)$, $\jelly\circ D$ is either zero or equal to $\pm\jelly_{D'}$ for some partition diagram $D':k\to 0$ having either $n$ or $n-1$ parts.  
\end{lemma}

\begin{proof}
    If two of the top vertices of $D$ are in the same part, then $\jelly\circ D$ will have two connected jellyfish legs, whence $\jelly\circ D=0$. Thus it suffices to consider the case where each of the $n$ top vertices of $D$ is in a separate part. Since we assume $D$ has at most $n$ parts, it follows that $D$ has exactly $n$ parts. Now, if two of the top vertices of $D$ are isolated (i.e.~not connected to any other vertex), then $\jelly\circ D$ will have two dangling legs, which implies $\jelly\circ D=0$. Hence we may assume that $D$ has at most one isolated vertex in the top row. In this case the bottom vertices of $D$ must be partitioned into $n$ or $n-1$ parts, each of which is connected to one of the non-isolated top vertices of $D$. Let $D':k\to 0$ denote the partition diagram whose vertices are in the same part precisely when the corresponding bottom vertices of $D$ are in the same part. Then $D'$ has $n$ or $n-1$ parts. Moreover, $\jelly_{D'}$ can be obtained from $\jelly\circ D$ by permuting some of the jellyfish legs. Thus $\jelly\circ D=\pm\jelly_{D'}$ by (\ref{skew form relation})
\end{proof}

\begin{proposition}\label{basis for JP}
The set $\{D\}_D\cup\{j_D\}_D$ is a basis for the space $\Hom_{\JP(n)}(k,0)$ where the first (resp.~second) set is indexed by all partition diagrams $D:k\to0$ having at most $n$ parts (resp.~having $n$ or $n-1$ parts).
\end{proposition}

\begin{proof} By Proposition \ref{prop: f pm basis}, the set $\{D\}_D\cup\{j_D\}_D$ described in this proposition has size $\dim_\k\Hom_{A_n}(V^{\otimes k},\unit)$. Moreover, $\dim_\k\Hom_{\JP(n)}(k,0)\geq\dim_\k\Hom_{A_n}(V^{\otimes k},\unit)$ by Theorem \ref{Psi is full}. Thus, it suffices to show $\{D\}_D\cup\{j_D\}_D$ spans $\Hom_{\JP(n)}(k,0)$. Every morphism in $\Hom_{\JP(n)}(k,0)$ equals a linear combination of jellyfish diagrams. By (\ref{jellyfish relation}) it suffices to show that every jellyfish diagram with at most one jellyfish is in the desired span. By Lemma \ref{span partition diagrams} all partition diagrams are in the desired span. On the other hand, every jellyfish diagram of type $k\to 0$ with exactly one jellyfish can be written as $j\circ D$ for some partition diagram $D:k\to n$. By Lemma \ref{span partition diagrams}, $D$ is equal to a linear combination of partition diagram having at most $n$ parts. Thus, it suffices to show $j\circ D$ is in the desired span whenever $D:k\to n$ is a partition diagram with at most $n$ parts. This follows from Lemma \ref{jelly-D span}.   
\end{proof}

\subsection{Proof of Theorem \ref{Psi is faithful}}\label{proof of main result}
As a consequence of Propositions \ref{prop: f pm basis} and \ref{basis for JP} we have 
\[\dim_\k\Hom_{\JP(n)}(k,0)=\dim_\k\Hom_{A_n}(V^{\otimes k},\unit)\]
for all $k\geq0$. Using the $\k$-linear isomorphisms given by the horizontal maps in (\ref{Reduce to invariants alt}) it follows that 
\[\dim_\k\Hom_{\JP(n)}(k,\ell)=\dim_\k\Hom_{A_n}(V^{\otimes k},V^{\otimes \ell})\]
for all $k,\ell\geq0$. Thus, the fullness of $\Psi$ (Theorem \ref{Psi is full}) implies $\Psi$ is also faithful. \qed

\bibliographystyle{alphanum}    
\bibliography{references}   

\end{document}